\documentclass[twoside]{aiml20}

\usepackage{aiml20macro}

\usepackage{graphicx}
\usepackage{amsmath}
\usepackage{amssymb}

\usepackage{latexsym}

\usepackage{verbatim}
\usepackage{setspace}
\usepackage{pbox}
\usepackage{enumerate}

\usepackage{bussproofs}

\usepackage{graphicx}
\usepackage{tikz}
\usetikzlibrary{tikzmark}

\usepackage{bpextra}

{\usepackage{txfonts}  
   \DeclareSymbolFont{symbolsC}{U}{txsyc}{m}{n}
   \DeclareMathSymbol{\strictif}{\mathrel}{symbolsC}{74}
   \DeclareMathSymbol{\boxright}{\mathrel}{symbolsC}{128}
}

\usetikzlibrary{arrows,decorations.pathmorphing,backgrounds,positioning,fit}

\newtheorem{df}{Definition}[section]  
\newtheorem{teo}[df]{Theorem}
\newtheorem{coro}[df]{Corollary}
\newtheorem{lm}[df]{Lemma}
\newtheorem{exm}[df]{Example}

\newtheorem{prp}[df]{Proposition}

\newenvironment{lateproof}[1] {\paragraph{Proof of {#1}.}}{\hfill$\square$}

\usepackage{bbm}

\usepackage{array}
\newcolumntype{C}[1]{>{\centering\arraybackslash}p{#1}}
\newcolumntype{L}[1]{>{\arraybackslash}p{#1}}
\usepackage{multirow}
\usepackage{float}

\newcommand{\dep}[2]{=\hspace{-3pt}({#1};{#2})}
\newcommand{\depc}[1]{=\hspace{-3pt}({#1})}
\newcommand{\con}[1]{=\hspace{-3pt}({#1})}

\newcommand{\ded}{\vdash_\sigma\xspace}

\newcommand{\cf}{\boxright}

\newcommand{\dblsetminus}{\mathbin{{\setminus}\mspace{-5mu}{\setminus}}}

\newcommand{\COV}{\ensuremath{\mathcal{CO}_{\mathsmaller{\dblsetminus\hspace{-0.23ex}/}}[\sigma]}}
\newcommand{\COv}{\ensuremath{\mathcal{CO}_{{\setminus}\mspace{-5mu}{\setminus}\hspace{-0.23ex}/}}}
\newcommand{\COd}{\ensuremath{\mathcal{COD}}}
\newcommand{\COD}{\ensuremath{\mathcal{COD}[\sigma]}}

\newcommand{\CO}{\mathcal{CO}[\sigma]}
\newcommand{\Co}{\mathcal{CO}}

\newcommand{\F}{\mathcal{F}}
\newcommand{\G}{\mathcal{G}}
\newcommand{\K}{\mathcal{K}}

\newcommand{\End}{\mathrm{En}}
\newcommand{\Exo}{\mathrm{Ex}}
\newcommand{\Con}{\mathrm{Cn}}

\newcommand{\Dom}{\mathrm{Dom}}
\newcommand{\Ran}{\mathrm{Ran}}
\newcommand{\ASS}{\ensuremath{\mathbb{A}_\sigma}}
\newcommand{\FUN}{\ensuremath{\mathbb{F}_\sigma}}
\newcommand{\FUNr}{\ensuremath{\mathbb{F}^0_\sigma}}
\newcommand{\SEM}{\ensuremath{\mathbb{S}\mathbbm{em}_\sigma}}

\newcommand{\CT}{\ensuremath{\mathbb{C}_\sigma}}

\newcommand{\SET}[1]{\mathbf{#1}}

\usepackage{relsize}

\usepackage{xspace}

\newcommand{\vvee}{\raisebox{1pt}{\ensuremath{\mathop{\,\mathsmaller{\dblsetminus\hspace{-0.23ex}/}\,}}}}
\newcommand{\bigvvee}{\ensuremath{\mathop{\mathlarger{\mathlarger{\mathbin{{\setminus}\mspace{-5mu}{\setminus}}\hspace{-0.33ex}/}}}}\xspace}

\newcommand{\todob}[1]{
\textnormal{\color{blue}\scriptsize+++Fausto: #1+++}}

\def\lastname{Barbero, Yang}

\begin{document}

\begin{frontmatter}
  \title{Counterfactuals and dependencies on causal teams: expressive power and deduction systems}
  \author{Fausto Barbero}\footnote{The author was supported by grant 316460 of  the Academy of Finland.}
  \address{Department of Philosophy, University of Helsinki \\ PL 24 (Unioninkatu 40), 00014  Helsinki, Finland}
  \author{Fan Yang}\footnote{The  author was supported by Research Funds of the University of Helsinki and grant 308712 of the Academy of Finland.}
  \address{Department of Mathematics and Statistics, University of Helsinki\\ PL 68 (Pietari Kalmin katu 5), 00014  Helsinki, Finland}
  
  \begin{abstract}
 We analyze the  causal-observational languages that were introduced in Barbero and Sandu (2018), which allow discussing interventionist counterfactuals and functional dependencies in a unified framework.
In particular, we systematically investigate the expressive power of these languages in causal team semantics, and we provide complete natural deduction calculi for each language. 
As an intermediate step towards the completeness, we axiomatize the languages over a generalized version of causal team semantics, 
which turns out to be interesting also in its own right.
  \end{abstract}

  \begin{keyword}
 Interventionist counterfactuals, causal teams, dependence logic, team semantics.
  \end{keyword}
 \end{frontmatter}

\section{Introduction}

Counterfactual conditionals express the modality of \emph{irreality}: they describe what \emph{would} or \emph{might} be the case in circumstances which diverge from the actual state of affairs. Pinning down the exact meaning and logic of counterfactual statements has been the subject of a large literature (see e.g. \cite{Sta2019}). We are interested here in a special case: the \emph{interventionist} counterfactuals, which emerged from the literature on causal inference (\cite{SpiGlySch1993,Pea2000,Hit2001}). Under this reading, a conditional $\SET X=\SET x\cf\psi$ says that $\psi$ would hold if we were to intervene on the given system, by subtracting the variables $\SET X$ to their current causal mechanisms and forcing them to take the values $\SET x$.

The \emph{logic} of interventionist counterfactuals has been mainly studied in the semantical context of \emph{deterministic causal models} (\cite{GalPea1998,Hal2000,Bri2012,Zha2013}), which consist of an assignment of values to variables together with a system of \emph{structural equations} that describe the causal connections. 
In \cite{BarSan2018}, causal models were generalized to \emph{causal teams}, in the spirit of \emph{team semantics} (\cite{Hod1997,Vaa2007}), by allowing a \emph{set} of assignments (a ``{\em team}'')  instead of a single assignment. This 
opens the possibility of describing e.g. \emph{uncertainty}, \emph{observations}, and \emph{dependencies}.

One of the main reasons for introducing causal teams was the possibility of comparing the logic of dependencies of causal nature (those definable in terms of interventionist counterfactuals) against that of contingent dependencies (such as have been studied in the literature on team semantics, or in database theory) in a unified semantic framework.  \cite{BarSan2018} and \cite{BarSan2019} give anecdotal evidence of the interactions between the two kinds of dependence, but offer no general axiomatizations for languages that also involve contingent dependencies. 
In this paper we fill this gap in the literature by providing complete deduction systems (in natural deduction style) for the languages $\COd$ and $\COv$ (from \cite{BarSan2018}), which 
enrich the basic counterfactual language, respectively, with atoms of functional dependence $\dep{\SET X}{Y}$ (``$Y$ is functionally determined by $\SET X$''), or with the intuitionistic disjunction $\vvee$, in terms of which functional dependence is definable.  We also give semantical characterizations, for $\COd$, $\COv$ and the  basic counterfactual $\Co$, in terms of definability of classes of causal teams.

The strategy of the completeness proofs is the following. We introduce a \emph{generalized} causal team semantics, which encodes uncertainty over causal models, not only over assignments. (This semantics is used as a tool towards completeness, but also has independent interest.) We then give completeness results for this semantics, using techniques developed in \cite{YanVaa2016,Cia2016b}. Finally,  we extend the calculi to completeness over causal teams by adding axioms which 
capture the property of being a causal team (i.e. encoding \emph{certainty} about the causal connections). 


The paper is organized as follows.
Section \ref{SECSYNSEM} introduces the formal languages and two kinds of semantics. Section \ref{SECCHARFUN} deepens the discussion of the functions which describe causal mechanisms, addressing issues of definability and the treatment of dummy arguments. 
Section \ref{SECCODED} characterizes semantically the language $\Co$ and reformulates in natural deduction form the $\Co$ calculi that come from \cite{BarSan2019}.  Section 
\ref{SECCOMPLETE} gives semantical characterizations for $\COd$ and $\COv$, and complete natural deduction calculi for both kinds of semantics. 


		


\section{Syntax and semantics} \label{SECSYNSEM}

\subsection{Formal languages}\label{SECLANG}

Let us start by fixing the syntax. 
 Each of the languages considered in this paper is parametrized by a (finite) \textbf{signature} $\sigma$, i.e. 
 a pair $(\Dom,\Ran)$, where $\Dom$ is a nonempty finite set of \textbf{variables}, and $\Ran$ is a function that associates to each variable $X\in \Dom$ a nonempty finite set $\Ran(X)$ (called the \textbf{range} of $X$) of \textbf{constant symbols} or \textbf{values}.\footnotemark\footnotetext{Note that we do not encode a distinction between exogenous and endogenous variables into the signatures, as done in \cite{Hal2000}. Instead, we follow the style of Briggs \cite{Bri2012}. Doing so will result in more general completeness results. 
 }
%
%
We reserve the Greek letter $\sigma$ for signatures. We use a boldface capital letter $\SET X$ to stand for a sequence $\langle X_1,\dots,X_n\rangle$ of variables; similarly a boldface lower case letter $\SET x$ stands for a sequence $\langle x_1,\dots,x_n\rangle$ of values. We will sometimes abuse notation and treat $\SET X$ and $\SET x$ as sets.

An atomic $\sigma$-formula is an equation $X=x$, where $X\in \Dom$ and $x\in \Ran(X)$. The conjunction $X_1=x_1 \land \dots \land X_n=x_n$ of equations is abbreviated as $\SET X = \SET x$, 
and also called an equation.  
Compound formulas of the basic language $\CO$ are formed by the grammar:
\begin{center}$\alpha::= X=x \mid   \neg \alpha   \mid  \alpha\land\alpha   \mid  \alpha\lor\alpha   \mid  
 \SET X=\SET x\cf\alpha$\end{center}
The connective $\cf$ is used to form \emph{interventionist counterfactuals}. 
We abbreviate $\neg(X=x)$ as $X\neq x$, and $X=x\wedge X\neq x$ as $\bot$. Throughout the paper, we 
reserve the first letters of the Greek alphabet, $\alpha,\beta,\dots$ for $\CO$ formulas. 

We consider also two extensions of $\CO$, obtained by adding the  {\em intuitionistic disjunction} $\vvee$, 
 or dependence atoms $\dep{\SET X}{Y}$: 
\begin{itemize}
\item $ \COV:\, \varphi::= X=x \mid   \neg \alpha   \mid  \varphi\land\varphi   \mid  \varphi\lor\varphi \mid \varphi\vvee\varphi  \mid 
\SET X=\SET x\cf\varphi $

\item $\COD: \,\varphi::=X=x \mid \ \dep{\SET X}{Y} \mid  \neg \alpha   \mid  \varphi\land\varphi   \mid  \varphi\lor\varphi   \mid  
 \SET X=\SET x\cf\varphi$

\end{itemize}

%


\subsection{Causal teams} \label{SUBSCT}

We now define the team semantics of our logics over causal teams. We first recall the definition of causal teams adapted from \cite{BarSan2019}.




Fix a signature $\sigma=(\Dom,\Ran)$. An \textbf{assignment} over  $\sigma$ 
is a mapping
$s:\Dom\rightarrow\bigcup_{X\in \Dom}\Ran(X)$ such that $s(X)\in \Ran(X)$
for each $X\in \Dom$.\footnote{We identify syntactical variables and values with their semantical counterpart, following the conventions in most literature on interventionist counterfactuals, e.g. \cite{GalPea1998,Hal2000,Bri2012,Zha2013}. In this convention distinct symbols (e.g., $x,x'$) denote distinct objects. 
} Denote by $\ASS$ the set of all assignments over $\sigma$. 
A \textbf{team} $T$ over $\sigma$ 
is a set of assignments over $\sigma$, i.e., $T\subseteq \ASS$.


A \textbf{system of functions} $\F$ over $\sigma$ 
is a function  that assigns to each variable $V$ in a domain $\End(\F)\subseteq \Dom$ a set $PA^{\F}_V\subseteq \Dom\setminus\{V\}$ of  \textbf{parents} of $V$, and a function  $\F_V: \Ran(PA^{\F}_V)\rightarrow \Ran(V)$.\footnote{We  identify the set $PA^{\F}_V$ with a sequence, in a fixed lexicographical ordering.} Variables in the set $\End(\F)$ are called \textbf{endogenous variables}  of $\F$, and variables in $\Exo(\F)=\Dom\setminus \End(\F)$ are called \textbf{exogenous variables} of $\F$. 

Denote by $\FUN$ the set of all systems of functions over $\sigma$, which is clearly finite. We say that an assignment $s\in \ASS$ is \textbf{compatible} with a system  of functions $\mathcal F\in \FUN$ if for all endogenous variables $V\in \End(\F)$, 
$s(V)=\F_V(s(PA_V^{\F}))$.





\begin{df}
A \textbf{causal team} over a signature $\sigma$  is a pair $T = (T^-,\F)$ consisting of \begin{itemize}
\item a team $T^-$ over $\sigma$, called the \textbf{team component} of $T$,
\item and a system of functions $\F$ over $\sigma$, called the \textbf{function component} of $T$,
\end{itemize}
where all assignments $s\in T^-$ are compatible with the function component $\F$.
\end{df}


Any system $\F\in \FUN$  of functions can be naturally associated with a (directed) graph $G_{\F}=(\Dom,E_{\F})$, defined as 
$(X,Y)\in E_\F$ iff $X\in PA_Y^\F.$
We say that  $\F$ is \textbf{recursive} if $G_\F$ is acyclic, i.e., for all $n\geq 0$, $E_{\F}$ has no subset of the form $\{(X_0,X_1),(X_1,X_2),\dots, (X_{n-1},X_n), (X_n,X_0)\}$. The graph of a causal team $T$, denoted as $G_T$,  is the associated graph of its function component. We call $T$ \textbf{recursive} if  $G_{T}$ is acyclic. Throughout this paper, for simplicity we assume that all causal teams that we consider are recursive.


Intuitively, a causal team $T$  may be seen as representing an assumption concerning the causal relationships among the variables in $\Dom$ (as encoded in $\F$) together with a range of hypotheses concerning the actual state of the system (as encoded in $T^-$). We now illustrate this idea in the following example.

\begin{exm}\label{EXCT}
The following diagram illustrates a causal team $T=(T^-,\F)$. 
\begin{center}
$T^-$: \begin{tabular}{|c|c|c|c|}
\hline
 \multicolumn{4}{|l|}{ } \\[-6pt]
 \multicolumn{4}{|l|}{
U\tikzmark{FROMU}  \, \tikzmark{TOX}X\tikzmark{FROMX} \,  \tikzmark{TOY}Y\tikzmark{FROMY}   \, \   \tikzmark{TOZ}Z} \\
\hline
 $0$ & $0$ & $1$ & $2$ \\
\hline
 $1$ & $1$ & $2$ & $6$ \\
\hline
\end{tabular}
 \begin{tikzpicture}[overlay, remember picture, yshift=.25\baselineskip, shorten >=.5pt, shorten <=.5pt]
	\draw [->] ([yshift=3pt]{pic cs:FROMU})  [line width=0.2mm] to ([yshift=3pt]{pic cs:TOX});
	\draw [->] ([yshift=3pt]{pic cs:FROMX})  [line width=0.2mm] to ([yshift=3pt]{pic cs:TOY});
	\draw [->] ([yshift=3pt]{pic cs:FROMY})  [line width=0.2mm] to ([yshift=3pt]{pic cs:TOZ});
  
	\draw ([yshift=7pt]{pic cs:FROMU})  edge[line width=0.2mm, out=30,in=160,->] ([yshift=8pt]{pic cs:TOZ});
	\draw ([yshift=8pt]{pic cs:FROMX})  edge[line width=0.2mm, out=20,in=160,->] ([yshift=5pt]{pic cs:TOZ});
  \end{tikzpicture}
	\hspace{30pt} 
	$\left\{
	\begin{array}{lcl}
	\F_X(U) & = & U \\
	\F_Y(X) & = & X+1 \\
	\F_Z(X,Y,U) & = & 2*Y+X+U \\
	\end{array}
	\right.$
\end{center}
The table on the left represents a team $T^-$ consisting of two assignments, each of which is tabulated in the obvious way as a row in the table. For instance, the assignment $s$ of the first row is defined as $s(U)=s(X)=0$, $s(Y)=1$ and $s(Z)=2$. The arrows in the upper part of the table represent the graph $G_{T}$ of the causal team $T$. 
 For instance, the arrow from $U$ to $Z$ represents the edge $(U,Z)$ in  $G_{T}$. The graph contains no cycles, thus the causal team $T$ is recursive. The variable $U$ with no incoming arrows is an exogenous variable.
 The other variables are endogenous variables, namely, $\End(\F)= \{X,Y,Z\}$. The function component is determined by the system of functions on the right of the above diagram. Each equation defines the ``law'' that generates the values of an endogenous variable. 


\end{exm}


Let $S=(S^-,\F)$ and $T=(T^-,\G)$ be causal teams over the same signature. We call $S$ a \textbf{causal subteam} of $T$, denoted as $S\subseteq T$, if $S^-\subseteq T^-$ 
and  $\F = \G$.


An equation  $\SET X = \SET x$ is said to be \textbf{inconsistent} if it contains two conjuncts  $X=x$ and $X=x'$ with distinct values $x,x'$; otherwise it is said to be \textbf{consistent}. 

\begin{df}[Intervention]\label{intervention_ct_df}
Let $T=(T^-,\mathcal F)$ be a causal team over some signature $\sigma=(\Dom,\Ran)$. Let $\SET X=\SET x\,(=X_1=x_1\wedge\dots\wedge X_n=x_n)$ be a consistent equation over $\sigma$. The \textbf{intervention} $do(\SET X= \SET x)$ on $T$ is the procedure that generates a new causal team $T_{\SET X = \SET x}=(T_{\SET X = \SET x}^-,\mathcal F_{\SET X=\SET x})$ over $\sigma$ defined as follows:
\begin{itemize}
\item $\F_{{\SET X = \SET x}}$ is the restriction of $\F$ to $\End(\F)\setminus \SET X$,

\item $T_{\SET X=\SET x}^-=\{s_{\SET X=\SET x}\mid s\in T^-\}$, where each $s_{\SET X=\SET x}$ is an assignment compatible with $\mathcal F_{{\SET X = \SET x}}$ defined (recursively) as 
\begin{center}\(s_{\SET X=\SET x}(V)=\begin{cases}
x_i&\text{ if }V=X_i,\\
s(V)&\text{ if }V\notin \End(T) \cup \SET X,\\
\F_V(s_{\SET X=\SET x}(PA_V^{\mathcal F}))&\text{ if }V\in \End(T)\setminus \SET X
\end{cases}\)\end{center}
\end{itemize}
%
%
%
\end{df}






\begin{exm} \label{EXCTINT}
Recall the recursive causal team $T$ in Example \ref{EXCT}. 
By applying the intervention $do(X=1)$ to $T$, we obtain a new causal team $T_{X=1}=(T_{X=1}^-,\F_{X=1})$  as follows. The function component $\F_{X = 1}$ is determined by the 
equations:
\begin{center}
$\left\{
	\begin{array}{lcl}
		(\F_{X=1})_Y(X) & = & X+1 \\
	(\F_{X=1})_Z(X,Y) & = & 2*Y+X+U \\
	\end{array}
	\right.$
\end{center}
The endogenous variable $X$ of the original team $T$ becomes  exogenous in the new team  $T_{X=1}$, and the equation $\F_X(U) = U$ for $X$ is now removed.
The new team component $T_{X=1}^-$ is obtained by the rewriting procedure illustrated below:
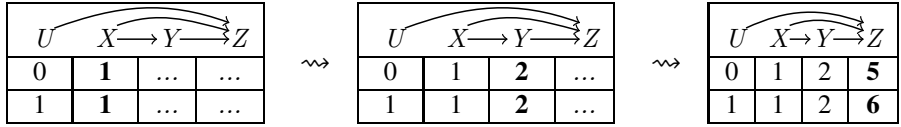
\begin{center}
 \begin{tabular}{|c|c|c|c|}
\hline
 \multicolumn{4}{|l|}{ } \\[-4pt]
 \multicolumn{4}{|l|}{
 \ U\tikzmark{FROMU1}  \, \ \ \tikzmark{TOX1}X\tikzmark{FROMX1} \, \ \ \tikzmark{TOY1}Y\tikzmark{FROMY1}   \, \ \ \  \tikzmark{TOZ1}Z} \\
\hline
 \ $0$ \ & \ $\mathbf{1}$ \ & ... & ... \\
\hline
 $1$ & $\mathbf{1}$ & ... & ... \\
\hline
\end{tabular}
 \begin{tikzpicture}[overlay, remember picture, yshift=.25\baselineskip, shorten >=.5pt, shorten <=.5pt]
 
	\draw [->] ([yshift=3pt]{pic cs:FROMX1})  [line width=0.2mm] to ([yshift=3pt]{pic cs:TOY1});
	\draw [->] ([yshift=3pt]{pic cs:FROMY1})  [line width=0.2mm] to ([yshift=3pt]{pic cs:TOZ1});

		\draw ([yshift=7pt]{pic cs:FROMU1})  edge[line width=0.2mm, out=20,in=160,->] ([yshift=9pt]{pic cs:TOZ1});
	\draw ([yshift=8pt]{pic cs:FROMX1})  edge[line width=0.2mm, out=20,in=165,->] ([yshift=6pt]{pic cs:TOZ1});
  \end{tikzpicture}
	\hspace{3pt} 
	$\rightsquigarrow$
	\hspace{3pt}
\begin{tabular}{|c|c|c|c|}
\hline
 \multicolumn{4}{|l|}{ } \\[-4pt]
 \multicolumn{4}{|l|}{
 \ U\tikzmark{FROMU2}  \, \ \ \tikzmark{TOX2}X\tikzmark{FROMX2} \, \ \ \tikzmark{TOY2}Y\tikzmark{FROMY2}   \, \ \ \  \tikzmark{TOZ2}Z} \\
\hline
 \ $0$ \ & \ $1$ \ & \ $\mathbf{2}$ \ & ... \\
\hline
 $1$ & $1$ & $\mathbf{2}$ & ... \\
\hline
\end{tabular}
 \begin{tikzpicture}[overlay, remember picture, yshift=.25\baselineskip, shorten >=.5pt, shorten <=.5pt]
 
	\draw [->] ([yshift=3pt]{pic cs:FROMX2})  [line width=0.2mm] to ([yshift=3pt]{pic cs:TOY2});
	\draw [->] ([yshift=3pt]{pic cs:FROMY2})  [line width=0.2mm] to ([yshift=3pt]{pic cs:TOZ2});

	\draw ([yshift=7pt]{pic cs:FROMU2})  edge[line width=0.2mm, out=20,in=160,->] ([yshift=9pt]{pic cs:TOZ2});
	\draw ([yshift=8pt]{pic cs:FROMX2})  edge[line width=0.2mm, out=20,in=165,->] ([yshift=6pt]{pic cs:TOZ2});

  \end{tikzpicture}
	\hspace{3pt} 
	$\rightsquigarrow$
	\hspace{3pt}
\begin{tabular}{|c|c|c|c|}
\hline
 \multicolumn{4}{|l|}{ } \\[-4pt]
 \multicolumn{4}{|l|}{
  U\tikzmark{FROMU3}  \,   \tikzmark{TOX3}X\tikzmark{FROMX3} \,   \tikzmark{TOY3}Y\tikzmark{FROMY3}   \,  \   \tikzmark{TOZ3}Z} \\
\hline
  $0$  &  $1$  &  $2$ & $\mathbf{5}$ \\
\hline
 $1$ & $1$ & $2$ & $\mathbf{6}$ \\
\hline
\end{tabular}
 \begin{tikzpicture}[overlay, remember picture, yshift=.25\baselineskip, shorten >=.5pt, shorten <=.5pt]
 
	\draw [->] ([yshift=3pt]{pic cs:FROMX3})  [line width=0.2mm] to ([yshift=3pt]{pic cs:TOY3});
	\draw [->] ([yshift=3pt]{pic cs:FROMY3})  [line width=0.2mm] to ([yshift=3pt]{pic cs:TOZ3});

	\draw ([yshift=7pt]{pic cs:FROMU3})  edge[line width=0.2mm, out=25,in=160,->] ([yshift=9pt]{pic cs:TOZ3});
	\draw ([yshift=8pt]{pic cs:FROMX3})  edge[line width=0.2mm, out=20,in=160,->] ([yshift=6pt]{pic cs:TOZ3});

  \end{tikzpicture}
\end{center}
In the first step,  rewrite the $X$-column with value $1$. Then, update (recursively) the other columns using the functions from $\F_{{X = 1}}$. In this step, only the columns that correspond to ``descendants'' of $X$ will be modified, and the order in which these columns should be updated is completely determined by the (acyclic) graph $G_{T_{X=1}}$ of $T_{X=1}$. Since the variable $X$ becomes exogenous after the intervention, all arrows pointing to $X$ have to be removed, e.g., the arrow from $U$ to $X$. 
We refer the reader to \cite{BarSan2019} for more details and justification for this rewriting procedure.

\end{exm}

\begin{df}\label{ct_semantics_df}
Let $\varphi$ be a formula of the language $\COV$ or  $\COD$, and $T=(T^-,\mathcal F)$ a causal team over $\sigma$. We define the satisfaction relation $T\models^c\varphi$ (or simply $T\models\varphi$) over causal teams inductively as follows: 
\begin{itemize}
\item $T\models X=x$ $\iff$ for all $s\in T^-$, $s(X)=x$.\footnote{Note once more that  the symbol $x$ is used as both  a syntactical and a semantical object.}
\item $T\models\, \dep{\SET X}{Y}$ $\iff$ for all $s,s'\in T^-$, $s(\SET X)=s'(\SET X)$ implies $s(Y)=s'(Y)$.
\item $T\models \neg\alpha$ $\iff$ for all $s\in T^-$, $(\{s\},\mathcal F)\not \models \alpha$.
\item $T\models \varphi\land \psi$ $\iff$  $T\models \varphi$ and $T\models \psi$.
\item $T\models \varphi\lor \psi$ $\iff$ there are two causal subteams
$T_1,T_2$ of $T$ such that $T_1^-\cup T_2^- = T^-$, $T_1\models \varphi$ and $T_2\models \psi$.
\item $T\models\varphi\vvee\psi$ $\iff$ $T\models \varphi$ or $T\models\psi$.
\item $T\models \SET X = \SET x \cf \varphi$ $\iff$ $\SET X = \SET x$ is inconsistent or $T_{\SET X=\SET x}\models\varphi$.
\end{itemize}
\end{df}

We write a dependence atom $\dep{}{X}$ with an empty first component  as $\con{X}$. The semantic clause for $\con{X}$ reduces to:
\begin{itemize}
\item $T\models\, \con{X}$ ~~ iff ~~ for all $s,s'\in T^-$, $s(X)=s'(X)$.
\end{itemize}
Intuitively, the atom $\con{X}$ states that $X$ has a constant value in the team. 
It is easy to verify that dependence atoms are definable in $\COV$: 
\begin{center}\hfill$\displaystyle\depc{Y}\equiv \bigvvee_{y\in \Ran(Y)}Y=y~\text{ and }~\dep{\SET X}{Y}\equiv\bigvee_{\SET x\in \Ran(\SET X)}(\SET X=\SET x\,\wedge \depc{Y}).$\hfill (1)\end{center}

\noindent The \emph{selective implication} $\alpha \supset\varphi$ from \cite{BarSan2018} is now definable as $\neg \alpha\lor \varphi$. 
Its semantic clause reduces to:
\begin{itemize}
    \item $T \models \alpha\supset \psi  \iff T^\alpha \models \psi$, where $T^\alpha$ is the (unique) causal subteam of $T$ with team component $\{s \in T^- \mid \{s\}\models \alpha\}$.
\end{itemize}

\begin{exm}
Consider the causal team $T$ and the intervention $do(X=1)$ from Examples \ref{EXCT} and \ref{EXCTINT}. Clearly, $T_{X=1}\models Y=2$, and thus $T\models X=1 \cf Y=2$. 
We also have that $T\models\dep{Y}{Z}$, while $T_{X=1}\not\models\dep{Y}{Z}$ (contingent dependencies are not in general preserved by interventions). Observe that $T\models Y\neq 2 \lor Y=2$, while $T\not\models Y\neq 2 \vvee Y=2$.

\end{exm}

\subsection{Generalized causal teams} \label{SUBSGCT}


Given a signature $\sigma$, write 
\(\SEM:=\{(s,\mathcal F)\in \ASS\times \FUN\mid s\text{ is compatible with }\mathcal F\}.\)
The pairs $(s,\mathcal F)\in \SEM$ can be easily identified with the \emph{deterministic causal models} (also known as \emph{deterministic structural equation models}) that are considered in the literature on causal inference (\cite{SpiGlySch1993},\cite{Pea2000}, 
etc.).
One can 
identify a causal team $T=(T^-,\mathcal F)$ with the set 
\begin{center}\(
T^g=\{(s,\mathcal F)\in \SEM\mid s\in T^-\}
\)
\end{center}
of deterministic causal models with a uniform  function component $\mathcal F$ throughout the team. In this section, we introduce a more general notion of causal team, called {\em generalized causal team}, where the function component $\mathcal F$ does not have to be constant thoroughout the team. 





\begin{df}
A \textbf{generalized causal team} $T$ over a signature $\sigma$ is a set of pairs $(s,\mathcal F)\in \SEM$, that is, $T\subseteq \SEM$.
\end{df}

 Intuitively,  a generalized causal team encodes uncertainty about  which causal model governs the variables in $\Dom$ - i.e., uncertainty both on the values of the variables and on the laws that determine them. Distinct elements $(s,\mathcal F), (t,\mathcal G)$ of the same generalized causal team may also disagree on what is the set of endogenous variables, or on whether the system is recursive or not. A generalized causal team is said to be \textbf{recursive} if, for each  pair $(s,\mathcal F)$ in the team, the associated graph $G_\F$ is recursive. In this paper we  only consider recursive generalized causal teams.

For any generalized causal team $T$, define the \textbf{team component} of $T$ to be the set
\(T^- := \{s \mid (s,\mathcal F) \in T \text{ for some } \mathcal F\}.\) 
%
A \textbf{causal subteam} of $T$ is a subset $S$ of $T$, denoted as $S\subseteq T$.
The \textbf{union} $S\cup T$ of two generalized causal teams $S,T$ is their set-theoretic union.

A causal team $T$ can be identified with the generalized causal team $T^g$, which has a constant function component  in all its elements. 
Conversely, if $T$ is a nonempty generalized causal team in which all elements  have the same function component $\mathcal F$, i.e., $T=\{(s,\mathcal F)\mid s\in T^-\}$, we can  
naturally identify $T$ with the causal team 
\begin{center}\(T^c=(T^-,\mathcal F).\)\end{center}
In particular, a singleton generalized causal team $\{(s,\mathcal F)\}$ corresponds to  a singleton causal team $(\{s\},\F)$. Applying a (consistent) intervention $do(\SET X = \SET x)$ on $(\{s\},\mathcal F)$ generates a causal team $(\{s_{\SET X = \SET x}\},\mathcal F_{\SET X = \SET x})$ as defined in Definition \ref{intervention_ct_df}.
We can then define the result of the intervention $do(\SET X = \SET x)$ on $\{(s,\mathcal F)\}$ to be the generalized causal team $(\{s_{\SET X = \SET x}\},\mathcal F_{\SET X = \SET x})^g=\{(s_{\SET X = \SET x},\mathcal F_{\SET X = \SET x})\}$.
Interventions on arbitrary generalized causal teams are defined as follows. 


\begin{df}[Intervention over generalized causal teams]
Let $T$ be a (recursive) generalized causal team, and $\SET X = \SET x$  a consistent equation over $\sigma$. The intervention $do(\SET X = \SET x)$ on $T$ generates the generalized causal team 
\(T_{\SET X = \SET x} :=\{(s_{\SET X = \SET x},\F_{\SET X = \SET x})\mid (s,\F)\in T\}.\)
\end{df}

%

\begin{df}
Let $\varphi$ be a formula of the language $\COV$ or  $\COD$, and $T$ a generalized causal team over $\sigma$. The satisfaction relation $T\models^g\varphi$ (or simply $T\models\varphi$) over generalized causal teams is defined in the same way as in Definition \ref{ct_semantics_df}, except for slight differences in the following clauses:
\begin{itemize}
\item $T\models^g \neg\alpha$ ~~ iff ~~ for all $(s,\mathcal F)\in T$, $\{(s,\mathcal F)\}\not \models \alpha$.
\item $T\models^g \varphi\lor \psi$ ~~ iff ~~ there are two generalized causal subteams $T_1,T_2$ of $T$ such that $T_1\cup T_2 = T$, $T_1\models \varphi$ and $T_2\models \psi$.
\end{itemize}
\end{df}
We list some closure properties for our logics over both causal teams and generalized causal teams in the next theorem, whose proof  is left to the reader, or  see \cite{BarSan2019} for the causal team case.


\begin{teo} \label{TEOGENDW}  
Let $T,S$ be (generalized) causal teams over some signature $\sigma$.
\begin{description}
\item[Empty team property:] If $T^-=\emptyset$, then $T\models\varphi$. 
\item[Downward closure:] If $T\models\varphi$ and $S\subseteq T$, then $S\models\varphi$.
\item[Flatness of $\Co$-formulas:] If $\alpha$ is a $\CO$-formula, then 

\noindent\(T\models \alpha \iff  (\{s\},\F)\models^c\alpha \text{ for all }s\in T^-
~(\text{resp. } \{(s,\F)\}\models^g\alpha \text{ for all }(s,\F)\in T).\)
\end{description}
\end{teo}

The team semantics over causal teams and that over generalized causal teams with a constant function component are essentially equivalent, in the sense of the next lemma, whose proof is left to the reader. 
\begin{lm}\label{LEMIDENTIFY}
\begin{enumerate}[(i)]
\item For any causal team $T$, we have that $T\models^c\varphi\iff T^g\models^g\varphi.$
\item For any nonempty generalized causal team $T$ with a unique function component, we have that $T\models^g\varphi\iff T^c\models^c\varphi.$ 
\end{enumerate}
\end{lm}

\begin{coro}\label{gct_sem_cons_2_ct}
For any set $\Delta\cup\{\alpha\}$  of $\CO$-formulas, $\Delta\models^{g}\alpha$ iff $\Delta\models^{c}\alpha$. 
\end{coro}
\begin{proof}
By Lemma \ref{LEMIDENTIFY}, 
$\{(s,\mathcal{F})\}\models^{g}\beta$ iff $(\{s\},\mathcal F)\models^{c}\beta$ for any $\beta\in \Delta\cup\{\alpha\}$. Thus, the claim follows from the flatness of $\CO$-formulas.
\end{proof}

\section{Characterizing function components}\label{SECCHARFUN}

\subsection{Equivalence of function components 
} \label{SUBSEQUIV}


Consider a binary function $f$ and an $(n+2)$-ary function $g$ defined as 
\(f(X,Y)=X+Y\text{ and }g(X,Y,Z_1,\dots,Z_n)=X+Y.\)
Essentially $f$ and $g$ are the same function: $Z_1,\dots, Z_n$  are dummy arguments of $g$. We now characterize this idea in the notion of two function components being equivalent up to dummy arguments.

\begin{df}\label{DEFSIMFG}
Let $\mathcal F,\mathcal G$ be two function components over  $\sigma = (\Dom,\Ran)$.
\begin{itemize}

\item Let $V\in \Dom$. Two functions $\F_V$ and $G_V$ are said to be equivalent up to dummy arguments, denoted as $\F_V \sim \G_V$, if 
for any $\SET x\in \Ran(PA_V^{\F}\cap PA_V^{\G})$, $\SET y\in \Ran(PA_V^{\F}\setminus PA_V^{\G})$ and  $\SET z \in \Ran(PA_V^{\G}\setminus PA_V^{\mathcal F})$, 
we have that \(\F_V(\SET x\SET y)= \G_V(\SET x\SET z)\) (where we assume w.l.o.g. the shown orderings of the arguments of the functions). 


\item 
Let $\Con(\F)$ denote the set of endogenous variables $V$ of $\F$ for which  $\F_V$ is a constant function, i.e., for some fixed $c\in \Ran(V)$, $\F_V(\SET p)=c\text{ for all }\SET p\in PA_V^{\F}$. 
We say that $\F$ and $\G$ are \textbf{equivalent up to dummy arguments}, denoted as $\mathcal F \sim \mathcal G$, if 
$\End(\mathcal F)\setminus \Con(\mathcal F) = \End(\mathcal G)\setminus \Con(\mathcal G)$, 
and $\F_V \sim \G_V$ holds for all $V\in \End(\mathcal F)\setminus \Con(\mathcal F)$. 
\end{itemize}
\end{df}

It is easy to see that $\sim$ is an equivalence relation. The next lemma shows that the relation $\sim$ is  preserved under interventions. 

\begin{lm}\label{LEMINTSIM1}
For any function components $\mathcal F,\mathcal G\in \FUN$ and consistent equation $\SET X=\SET x$ over $\sigma$, we have that
$\mathcal F \sim \mathcal G$  implies $\mathcal F_{\SET X = \SET x} \sim \mathcal G_{\SET X = \SET x}$.
\end{lm}

\begin{proof}
Suppose $\mathcal F \sim \mathcal G$. Then $\End(\mathcal F)\setminus \Con(\mathcal F) = \End(\mathcal G)\setminus \Con(\mathcal G)$. Observe that
\(\End(\mathcal F_{\SET X=\SET x})=\End(\mathcal F)\setminus\SET X\text{ and }\Con(\mathcal F_{\SET X=\SET x})=\Con(\mathcal F)\setminus \SET X;\)
and similarly for $\mathcal G$. It follows that 
\(
\End(\mathcal F_{\SET X=\SET x})\setminus \Con(\mathcal F_{\SET X=\SET x}) = \big(\End(\mathcal F)\setminus \Con(\mathcal F)\big)\setminus \SET X= \big(\End(\mathcal G)\setminus \Con(\mathcal G) \big)\setminus \SET X= \End(\mathcal G_{\SET X=\SET x})\setminus \Con(\mathcal G_{\SET X=\SET x}).
\)
On the other hand, for any $V\in \End(\mathcal F_{\SET X=\SET x})\setminus \Con(\mathcal F_{\SET X=\SET x})=\big(\End(\mathcal F)\setminus \Con(\mathcal F)\big)\setminus \SET X$, by the assumption, $(\F_{\SET X = \SET x})_V=\F_V\sim \G_V=(\G_{\SET X = \SET x})_V$.
\end{proof}

We now generalize the equivalence relation $\sim$ to the team level. Let us first consider causal teams.
Two causal teams $T=(T^-,\mathcal F)$ and $S=(S^-,\mathcal G)$ of the same signature $\sigma$ are said to be \textbf{similar}, denoted as $T\sim S$, if $\mathcal F \sim \mathcal G$. We say that $T$ and $S$ are \textbf{equivalent}, denoted as $T\approx S$, if $T\sim S$ and $T^- = S^-$. 

Next, we turn to generalized causal teams.
We call a generalized causal team $T$ a \textbf{uniform team} if $\mathcal F\sim \mathcal G$ for all $(s,\mathcal F),(t,\mathcal G)\in T$. By Lemma \ref{LEMINTSIM1}, we know that if $T$ is uniform, so is $T_{\SET X = \SET x}$, for any consistent equation $\SET X = \SET x$. 
For any generalized causal team $T$ with $(t,\mathcal F)\in T$,  write 
\(
T^\mathcal F := \{(s,\mathcal G)\in T \mid \mathcal G\sim \mathcal F \}.
\)
Two generalized causal teams $S$ and $T$ are said to be  \textbf{equivalent}, denoted as $S\approx T$, if   $(S^{\F})^-=(T^{\F})^-$  for all $\F \in \FUN$. 

\begin{teo}[Closure under causal equivalence]\label{PROPEQUIV} 
Let $T,S$ be two (generalized) causal teams over $\sigma$ such that $T\approx S$. 
We have that \(T\models\varphi\iff S\models\varphi.\)
\end{teo}
\begin{proof}
The theorem is proved by induction on $\varphi$. The case $\varphi=\,\SET X=\SET x\cf \psi$ follows  from the fact that
$T_{\SET X = \SET x}\approx S_{\SET X = \SET x}$
 (Lemma \ref{LEMINTSIM1}). The case $\varphi=\psi\vee\chi$ for causal teams follows directly from the induction hypothesis. We now give the proof for this case for generalized causal teams. We only prove the left to right direction (the other direction is symmetric). Suppose $T\models^g\psi\vee\chi$. Then there are $T_0,T_1\subseteq T$ such that $T=T_0\cup T_1$, $T_0\models^g\psi$ and $T_1\models^g\chi$. Consider $S_i=\{(s,\F)\in S\mid \{(s,\F)\}\approx\{(s,\G)\}\text{ for some }(s,\G)\in T_i\}$ ($i=0,1$). It is easy to see that $S_i\approx T_i$ ($i=0,1$) and $S=S_0\cup S_1$. By induction hypothesis we have that $S_0\models^g\psi$ and $S_1\models^g\chi$. Hence $S\models^g\psi\vee\chi$.
\end{proof}


Thus, none of our languages can tell apart causal teams which are equivalent up to dummy arguments. However, one might not be sure, a priori, that a given argument behaves as dummy for a specific function, as this might e.g. be unfeasible to verify if the variables range over large sets. For this reason, we are keeping this distinction in the semantics instead of quotienting it out.

\subsection{Characterizing function components}

For any function component $\mathcal F$ over some signature $\sigma$, define a $\CO$-formula

\begin{center}\(
\displaystyle\Phi^{\mathcal F}:= \bigwedge_{V\in \End(\mathcal F)} \eta_\sigma(V) \land \bigwedge_{V\in (\Dom\setminus \End(\mathcal F))\cup \Con(\F)} \xi_\sigma(V).
\)
\end{center}
$
\begin{array}{l}
 \text{where  }\eta_\sigma(V):= \bigwedge\big\{(\SET W = \SET w \land PA_V^{\mathcal F} = \SET p)\cf V = \mathcal F_V(\SET p)\\
 \quad\quad\quad\quad\quad\quad\quad\quad\quad\quad\mid \SET W = \Dom \setminus(PA_V^\mathcal F\cup \{V\}) ,~\SET w \in \Ran(\SET W) ,~ \SET p\in \Ran(PA_V^\mathcal F)\big\}\\
\end{array}$

\noindent $\begin{array}{l}
\text{and  }\xi_\sigma(V):= \bigwedge\big\{ V=v \supset (\SET W_V=\SET w \cf V=v)\\
\quad\quad\quad\quad\quad\quad\quad\quad\mid v\in \Ran(V),\SET W_V=\Dom\setminus\{V\}, \SET w\in \Ran(\SET W_V)\big\}.
\end{array}$

\noindent Intuitively, for each non-constant endogenous variable $V$ of $\mathcal F$, the formula $\eta_\sigma(V)$ specifies that all assignments in the (generalized) causal team $T$  in question behave exactly as required by the function $\mathcal F_V$. For each variable $V$ which, according to $\mathcal F$, is exogenous or generated by a constant function, the formula $\xi_\sigma(V)$ states that $V$ is not affected by interventions on 
 other variables. If $V\in\Con(\F)$, then $V$ has both an $\eta_\sigma$ and a $\xi_\sigma$ clause.  Overall, the formula $\Phi^{\mathcal F}$ is satisfied in a team $T$ if and only if every assignment in $T$ has a function component that is $\sim$-equivalent to $\mathcal F$. 
 This result is the crucial element for adapting the standard methods of team semantics to the causal context.


\begin{teo}\label{LEMPHIF}
Let $\sigma$ be a signature, and $\mathcal F\in \FUN$.
\begin{enumerate}[(i)]
\item For any  generalized causal team $T$ over $\sigma$,  we have that

\begin{center}\(
T\models^g \Phi^\mathcal F  \iff \text{for all } (s,\mathcal G)\in T: \mathcal G \sim \mathcal F.  
\)
\end{center}
\item For any nonempty causal team $T = (T^-, \G)$ over $\sigma$, we have that

\begin{center}\(
T\models^c \Phi^\mathcal F  \iff \G \sim \F.  
\) \end{center}
\end{enumerate}
\end{teo}
\begin{proof}
(i). $\Longrightarrow$: Suppose $T\models^g \Phi^\mathcal F$ and $(s,\mathcal G)\in T$. We show $\mathcal G \sim \mathcal F$.
$\End(\mathcal F)\setminus \Con(\mathcal F) \subseteq \End(\mathcal G)\setminus \Con(\mathcal G)$: For any $V\in \End(\mathcal F)\setminus \Con(\mathcal F)$, there are distinct $\SET p,\SET p'\in \Ran(PA_V^{\mathcal F})$ such that $\mathcal F_V(\SET p) \neq \mathcal F_V(\SET p')$. Since $T\models\eta_\sigma(V)$, for any $\SET w \in Ran(\SET W)$, we have that 
\begin{center}$
\begin{array}{l}
\{(s,\mathcal G)\} \models (\SET W = \SET w \land PA_V^\mathcal F = \SET p) \cf V=\mathcal F_V(\SET p),\\
\{(s,\mathcal G)\} \models (\SET W = \SET w \land PA_V^\mathcal F =\SET p') \cf V=\mathcal F_V(\SET p').
\end{array}
$
\end{center}
Thus,
\( s_{\SET W = \SET w \land PA_V^{\mathcal F} = \SET p}(V)=\mathcal F_V(\SET p)\neq \mathcal F_V(\SET p')=s_{\SET W = \SET w \land PA_V^{\mathcal F}= \SET p'}(V).\)
So, $V\notin \Con(\G)$, and furthermore $V$ is not exogenous (since the value of an exogenous variable is not affected by interventions on different variables). Thus, $V\in \End(G)\setminus \Con(\G)$.

$\End(\mathcal G)\setminus \Con(\mathcal G) \subseteq \End(\mathcal F)\setminus \Con(\mathcal F)$: For any $V\in \End(\mathcal G)\setminus \Con(\mathcal G)$, there are distinct $\SET p,\SET p'\in Ran(PA_V^{\mathcal G})$ such that $\mathcal G_V(\SET p) \neq \mathcal G_V(\SET p')$.  Now, if $V\notin \End(\mathcal F)\setminus \Con(\F)$, then $T\models\xi_\sigma(V)$. Let $v=s(V)$ and $\SET Z=\SET W_V \setminus PA_V^\mathcal{G}$. Since $\{(s,\mathcal G)\}\models V=v$ and $V\notin PA_V^{\mathcal G}$, for any  $\SET z \in \Ran(\SET Z)$, we have that 
\begin{center}$
\begin{array}{l}
\{(s,\mathcal G)\} \models (\SET Z = \SET z \land PA_V^\mathcal G = \SET p) \cf V=v,\\
\{(s,\mathcal G)\} \models (\SET Z = \SET z \land PA_V^\mathcal G =\SET p') \cf V=v.
\end{array}
$\end{center}
By the definition of intervention, we must have that
\( v=s_{\SET Z = \SET z \land PA_V^{\mathcal G} = \SET p}(V)=\mathcal G_V(\SET p)\neq \mathcal G_V(\SET p')=s_{\SET Z = \SET z \land PA_V^{\mathcal G}= \SET p'}(V)=v,\)
which is impossible. Hence, $V\in \End(\mathcal F)\setminus \Con(\F)$.


$\mathcal F_V \sim \mathcal G_V$ for any $V\in \End(\mathcal F)\setminus \Con(\mathcal F)$: For any $\SET x\in \Ran(PA_V^{\mathcal F}\cap PA_V^{\mathcal G})$, $\SET y\in \Ran(PA_V^{\mathcal F}\setminus PA_V^{\mathcal G})$ and  $\SET z \in \Ran(PA_V^{\mathcal G}\setminus PA_V^{\mathcal F})$,  since $T\models\eta_\sigma(V)$ and $V\notin PA_V^{\G}$, for any $\SET w\in \Ran(\SET W)$ with $\SET w\upharpoonright (PA_V^{\mathcal G}\setminus PA_V^{\mathcal F}) = \SET z$, we have that
\begin{center}\(\{(s,\mathcal G)\}\models (\SET W = \SET w \land PA_V^{\mathcal F} = \SET x\SET y)\cf V = \mathcal F_V(\SET x\SET y).\)
\end{center}
Then 
\( \mathcal F_V(\SET x\SET y)= s_{\SET W = \SET w \land PA_V^{\mathcal F} = \SET x\SET y}(V)=\mathcal G_V(s_{\SET W = \SET w \land PA_V^{\mathcal F}= \SET x\SET y}(PA_V^\mathcal G)) =\mathcal G_V(\SET x\SET z),\)
as required. 

$\Longleftarrow$: Suppose that $\mathcal G \sim \mathcal F$ for all $(s,\mathcal G)\in T$. Since the formula $\Phi^{\F}$ is flat, it suffices to show that $\{(s,\mathcal G)\}\models \eta_\sigma(V)$ for all $V\in \End(\mathcal F)$, and $\{(s,\mathcal G)\}\models \xi_\sigma(V)$ for all $V\in (\Dom\setminus \End(\mathcal F))\cup \Con(\F)$.

For the former, take any  $\SET w\in \Ran(\SET W)$ and $\SET p\in \Ran(PA_V^\mathcal F)$, and let  $\SET Z = \SET z$ abbreviate $\SET W = \SET w \land PA_V^{\mathcal F} = \SET p$. We show that $\{(s_{\SET Z = \SET z},\mathcal G_{\SET Z = \SET z})\}\models V=\mathcal F_V(\SET p)$. Since $\mathcal G \sim \mathcal F$, by Lemma \ref{LEMINTSIM1} we have that $ \mathcal G_{\SET Z = \SET z}\sim \mathcal F_{\SET Z = \SET z}$. Thus, 
\begin{center}$\begin{array}{rlr}
s_{\SET Z = \SET z}(V)= (\mathcal G_{\SET Z = \SET z})_V(s_{\SET Z = \SET z}(PA_V^{\mathcal G_{\SET Z = \SET z}}))&= (\mathcal F_{\SET Z = \SET z})_V(s_{\SET Z = \SET z}(PA_V^{\mathcal F_{\SET Z = \SET z}}))&(\text{since }\mathcal G_{\SET X= \SET x} \sim \mathcal F_{\SET X= \SET x})
\\
& = \mathcal F_V(s_{\SET Z = \SET z}(PA_V^{\mathcal F}))&(\text{since } V\notin \SET Z)\\
&=\mathcal F_V(\SET p).&
\end{array}
$\end{center}

For the latter, take any $v\in \Ran(V)$ and $\SET w\in \Ran(\SET W_V)$. Assume that $\{(s,\mathcal G)\}\models V=v$, i.e., $s(V)=v$. Since $V\notin \End(\mathcal F)\setminus \Con(\F)$ and $\mathcal F\sim\mathcal G$, we know that $V\notin \End(\mathcal G)$ or $V\in \Con(\mathcal G)$. In both cases we have that $\{(s,\mathcal G)\}\models \SET W_V=\SET w \cf V=v$.

(ii). Let $T$ be a nonempty causal team. Consider its associated generalized causal team $T^g$. The claim then follows from Lemma \ref{LEMIDENTIFY} and item (i). 
\end{proof}

\begin{coro}\label{LEMCHARCT}
For any generalized causal team  $T$ over some signature $\sigma$,

\begin{center}\(
\displaystyle T\models \bigvvee_{\mathcal F \in \FUN}\Phi^\mathcal F\iff  T \text{ is uniform}.  
\)\end{center}
\end{coro}

The intuituionistic disjunction $\vvee$ was shown to have the {\em disjunction property}, i.e., $\models \varphi\vvee\psi$ implies $\models\varphi$ or $\models\psi$, in propositional inquisitive logic (\cite{Cia2016b}) and  propositional dependence logic (\cite{YanVaa2016}).
It follows immediately from Theorem \ref{LEMPHIF} that the disjunction property of $\vvee$ fails in the context of causal teams, because  $\models^c\bigvvee_{\mathcal F \in \FUN}\Phi^\mathcal F$, whereas $\not\models^c\Phi^\mathcal F$ for any $\mathcal F \in \FUN$. Nevertheless, the intuitionistic disjunction  does admit the disjunction property over generalized causal teams.

\begin{teo}[Disjunction property]
\label{splitting_prop}
Let $\Delta$ be a set of $\CO$-formulas, and $\varphi,\psi$ be arbitrary formulas over $\sigma$. If $\Delta\models^{g}\varphi\vvee\psi$, then $\Delta\models^{g}\varphi$ or $\Delta\models^{g}\psi$. 
In particular, if $\models^{g}\varphi\vvee\psi$, then $\models^{g}\varphi$ or $\models^{g}\psi$.  
\end{teo}
\begin{proof}
Suppose $\Delta\not\models^{g}\varphi$ and $\Delta\not\models^{g}\psi$. Then there are two generalized causal teams $T_1,T_2$ such that $T_1\models \Delta$, $T_2\models \Delta$, $T_1\not\models\varphi$ and $T_2\not\models\psi$. Let $T:= T_1\cup T_2$. 
By flatness of $\Delta$, we have that $T\models\Delta$. On the other hand, by downwards closure, we have that $T\not\models\varphi$ and $T\not\models\psi$, and thus $T\not\models\varphi\vvee \psi$.
\end{proof}

\section{Characterizing $\Co$}\label{SECCODED}

In this section, we 
characterize the expressive power of $\Co$ over causal teams and present a system of natural deduction for $\Co$ that is sound and complete over both causal teams and generalized causal teams. 

\subsection{Expressivity}\label{SECCOEXPR}

In this subsection, we 
show that $\Co$-formulas capture the flat class of causal teams (up to $\approx$-equivalence). 
Our result is analogous to known characterizations of flat languages in propositional team semantics (\cite{YanVaa2017}), with a twist, given by the fact that only the unions of similar causal teams are reasonably defined. We define such unions as follows.

\begin{df}\label{union_ct_df}
Let $S=(S^-,\F),T=(T^-,\G)$ be two causal teams over the same signature $\sigma$ with $S\sim T$. The union of $S$ and $T$ is defined as the causal team $S\cup T=(S^-\cup T^-,\mathcal H)$ over $\sigma$, where 
\begin{itemize}
\item $\End(\mathcal H)=(\End(\F)\setminus \Con(F))\cap (\End(\G)\setminus \Con(G))$,

\item  and for each $V\in \End(\mathcal H)$, $PA^{\mathcal H}_V=PA^{\F}_V\cap PA^{\G}_V$, and
\(\mathcal H_V(\SET p)=\F_V(\SET p\SET x)\)
for any $\SET p\in PA_V^{\F}\cap PA_V^{\G}$ and $\SET x \in PA_V^{\F}\setminus PA_V^{\G}$. 
\end{itemize}
\end{df}
Clearly, $\mathcal H\sim\F\sim\G$ and thus $S\cup T\sim S\sim T $.




A formula $\varphi$ over 
$\sigma$ determines a class $\K_\varphi$ of 
causal teams defined as
\begin{center}\(\K_\varphi=\{T\mid T\models\varphi\}.\)\end{center}
We say that a formula $\varphi$ \textbf{defines} a class $\K$ of  
causal teams if $\K=\K_\varphi$.

\begin{df}
We say that  a class  $\K$ of 
causal teams over  $\sigma$ is
\begin{itemize}
\item \textbf{causally downward closed} if $T\in \K$ and $S\subseteq T$ imply $S\in\K$;

\item \textbf{closed under causal unions} if, whenever  $T_1,T_2\in\K$ and $T_1\cup T_2$ is defined, 
 $T_1\cup T_2\in \K$;

\item \textbf{flat} if $(T^-,\F)\in \K$ iff $(\{s\},\F)\in \K$ for all $s\in T^-$; 
\item \textbf{closed under  equivalence}  if $T\in \K$ and $T\approx T'$ imply $T'\in \K$. 
\end{itemize}
\end{df}
It is easy to verify that $\K$ is flat iff $\K$ is causally downward closed and closed under causal unions. Any nonempty downward closed class $\K$ of causal teams over $\sigma$  contains 
all causal teams over $\sigma$ with empty team component. 
The class $\K_\varphi$ is always nonempty as the teams with empty team component are always in $\K_\varphi$ (by Theorem \ref{TEOGENDW}). By Theorems \ref{TEOGENDW} and \ref{PROPEQUIV}, if $\alpha$ is a $\Co$-formula, then $\K_\alpha$ is flat and closed under equivalence.
The main result of this section is the following characterization theorem which gives also the converse direction.

\begin{teo}\label{TEOCHARCOCT}
Let $\K$ be a nonempty (finite) class of 
causal teams over some signature $\sigma$. Then $\K$ is definable by a $\CO$-formula if and only if $\K$ is flat 
and closed under equivalence.
\end{teo}

In order to prove the above theorem, we  introduce a $\Co$-formula $\Theta^T$, inspired by a similar one in \cite{YanVaa2016}, that defines the property ``having as team component a subset of $T^-$''. 
For each causal team $T$ over $\sigma=(\Dom,\Ran)$, define
\begin{center}\(\displaystyle
\Theta^T := \bigvee_{s\in T^-} \bigwedge_{V\in \Dom} V = s(V). 
\)\end{center}


\begin{lm} \label{LEMCHARSUB_ct}
\(S\models\Theta^T\) iff \(S^-\subseteq T^-\), for any  causal teams $S,T$ over $\sigma$.
\end{lm}
\begin{proof}
``$\Longrightarrow$": Suppose $S\models\Theta^T$ and $S=(S^-,\F)$. 
For any $s\in S^-$, by downward closure, we have that $(\{s\},\F)\models \Theta^T$, which means that for some $t\in T^-$, \((\{s\},\F)\models V = t(V)\) for all $V\in\Dom$.
Since $\{s\}$ and $\{t\}$ have the same signature, this implies that $s=t$, thereby $s\in T^-$.

``$\Longleftarrow$": Suppose $S^-\subseteq T^-$. Observe that 
$S\models\Theta^{S}$ and $\Theta^T=\Theta^S\vee \Theta^{T\setminus S}$.
Thus, we conclude $S\models \Theta^T$ by the empty team property.
\end{proof}

%
%

\begin{lm}\label{COROCHARSUB_ct}
Let $S=(S^-,\G)$ and $T=(T^-,\F)$ be  causal teams over $\sigma$ with $S^-,T^-\neq\emptyset$. Then 
\(
S\models\Theta^{T}\wedge \Phi^{\F} \iff S\approx R\subseteq T\text{ for some }R\text{ over }\sigma.
\)

\end{lm}
\begin{proof}
By Lemma \ref{LEMCHARSUB_ct} and Theorem \ref{LEMPHIF}, we have that $S\models\Theta^{T}\wedge \Phi^{\F}$ iff $S^-\subseteq T^-$ and $\G\sim \F$. It then suffices to show that the latter is equivalent to $S\approx R\subseteq T$ for some $R$. The right to left direction is clear; conversely, if $S^-\subseteq T^-$ and $\G\sim \F$, then  we can take $R=(S^-,\F)$. 
\end{proof}

Consider the quotient set $\FUN/_{\!\approx}$. For each equivalence class $[\F]\in \FUN/_{\!\approx}$ choose a unique representative $\F_0$. Denote by $\FUNr$ the set of all such representatives.

\begin{lateproof}{Theorem \ref{TEOCHARCOCT}}
It suffices to prove the direction ``$\Longleftarrow$".  For each $\F\in\FUNr$, let $\K^\F := \{(T^-,\G)\in\K \mid \G \sim \F\}$. Clearly $\K = \bigcup_{\F\in \FUNr}\K^\F$. Let $T_{\F}=\bigcup\K^\F$, which is well-defined as in Definition \ref{union_ct_df}. Since $\K$ is closed under causal unions,  $T_{\F}\in\K$. We may assume w.l.o.g. that $T_{\F}=(T_{\F}^-,\F)$.  
 Let
\begin{center}\(\displaystyle\varphi= \bigvee_{\F \in \FUNr}(\Theta^{T_{\F} }\land \Phi^{\F }).\)
\end{center}

It suffices to show that $\K_\varphi=\K$.
For any $S=(S^-,\G)\in \K$, there exists $\F\in\FUNr$ such that  $S\in \K^{\F}$. Let $R=(S^-,\F)$. Clearly, $S\approx R\subseteq \bigcup\K^\F=T_{\F}$, which by Lemma \ref{COROCHARSUB_ct} implies that $S\models \Theta^{T_{\F} }\land \Phi^{\F }$. Hence, $S\models\varphi$, namely $S\in \K_{\varphi}$.

Conversely, suppose $S=(S^-,\G)\in \K_\varphi$, i.e., $S\models \varphi$. Then for every $\F\in \FUNr$, there is $S_{\F}\subseteq S$ such that $S=\bigcup_{\F\in\FUNr}S_{\F}$ and $S_{\F}\models \Theta^{T_\F }\land \Phi^{\F }$. Thus, by Lemma \ref{COROCHARSUB_ct}, we obtain that $S_{\F}\approx R_{\F}\subseteq T_{\F}$ for some $R_{\F}$. In particular, we have that $S_{\F}=(S_{\F}^-,\G)\sim (T_{\F}^-,\F)=T_{\F}$, which gives $\G\sim\F$. But since no two distinct elements in  $\FUNr$ are $\sim$-similar to each other, and $S_\F \sim S$ for each $\F \in \FUNr$, this can only happen if $S^-_{\F}=\emptyset$ for all $\F\in \FUNr$ except one. Denote this unique element of $\FUNr$ by $\mathcal H$. 
Now, $S=S_{\mathcal H}\approx R_{\mathcal H}\subseteq T_{\F}\in \K$. Hence we conclude that $S\in \K$, as $\K$ is causally closed downward and closed under equivalence.
\end{lateproof}

\subsection{Deduction system}

The logic $\CO$ over (recursive) causal teams was axiomatized in \cite{BarSan2019} by means of a sound and complete Hilbert style deduction system. 
In this section, we present an equivalent system of natural deduction and show it to be sound and complete also over (recursive) generalized causal teams. 

\begin{df}\label{co-system-df}
The system of natural deduction for $\CO$ consists of the following rules: 

\begin{itemize}
\item (Parameterized) rules for value range assumptions:
\begin{center}
{\normalfont
\renewcommand{\arraystretch}{1.8}
\def\ScoreOverhang{0.5pt}
\def\defaultHypSeparation{\hskip .1in}
\begin{tabular}{|C{0.45\linewidth}C{0.45\linewidth}|}
\hline
 \AxiomC{} \noLine\UnaryInfC{} \RightLabel{$\textsf{ValDef}$}\UnaryInfC{$\bigvee_{x\in \Ran(X)} X=x$} \noLine\UnaryInfC{}\DisplayProof & \AxiomC{}\noLine\UnaryInfC{$X=x$} \RightLabel{$\textsf{ValUnq}$}\UnaryInfC{$X\neq x'$} \noLine\UnaryInfC{}\DisplayProof\\\hline
 \end{tabular}
 }
\end{center}

\item Rules for $\wedge,\vee,\neg$:

\begin{center}
{\normalfont
\renewcommand{\arraystretch}{1.8}
\begin{tabular}{|C{0.45\linewidth}C{0.45\linewidth}|}
\hline
 \AxiomC{$\varphi$}\AxiomC{$\psi$} \RightLabel{$\wedge\textsf{I}$}\BinaryInfC{$\varphi\wedge\psi$} \DisplayProof&
 \AxiomC{$\varphi\wedge\psi$} \RightLabel{$\wedge\textsf{E}$}\UnaryInfC{$\varphi$} \DisplayProof~ \AxiomC{$\varphi\wedge\psi$} \RightLabel{$\wedge\textsf{E}$}\UnaryInfC{$\psi$} \DisplayProof\\
 \multirow{2}{*}{\AxiomC{$\varphi$} \RightLabel{$\vee\textsf{I}$}\UnaryInfC{$\varphi\vee\psi$}\DisplayProof~\AxiomC{$\varphi$} \RightLabel{$\vee\textsf{I}$}\UnaryInfC{$\psi\vee\varphi$}\DisplayProof}&\AxiomC{$\varphi\vee\psi$}\AxiomC{$[\varphi]$}\noLine\UnaryInfC{$\vdots$}\noLine\UnaryInfC{$\alpha$} \AxiomC{}\noLine\UnaryInfC{$[\psi]$}\noLine\UnaryInfC{$\vdots$}\noLine\UnaryInfC{$\alpha$}\RightLabel{$\vee\textsf{E}$} \TrinaryInfC{$\alpha$}\noLine\UnaryInfC{}\DisplayProof\\
 \multicolumn{2}{|c|}{\AxiomC{$[\alpha]$}\noLine\UnaryInfC{$\vdots$}\noLine\UnaryInfC{$\bot$} \RightLabel{$\neg\textsf{I}$}\UnaryInfC{$\neg\alpha$} \noLine\UnaryInfC{}\DisplayProof
\quad\quad\AxiomC{$\alpha$}\AxiomC{$\neg\alpha$}\RightLabel{$\neg\textsf{E}$}\BinaryInfC{$\varphi$} \DisplayProof
\quad\quad \AxiomC{$[\neg\alpha]$}\noLine\UnaryInfC{$\vdots$}\noLine\UnaryInfC{$\bot$}\RightLabel{\textsf{RAA}}\UnaryInfC{$\alpha$} \noLine\UnaryInfC{}\DisplayProof}\\\hline
\end{tabular}
}
\end{center}

\item Rules for $\cf$: 
{\normalfont
\begin{center}
\renewcommand{\arraystretch}{1.8}
\def\ScoreOverhang{0.5pt}
\def\defaultHypSeparation{\hskip .1in}
\begin{tabular}{|L{0.455\linewidth}L{0.455\linewidth}@{}|}
\hline\multicolumn{2}{|@{}l@{}|}{}\\[-18pt]
 \multicolumn{2}{|@{}l@{}|}{\!\AxiomC{} \RightLabel{$\cf\!\textsf{Eff}$}\UnaryInfC{$(\SET{X}=\SET{x} \land Y=y) \cf Y=y$} \DisplayProof
 \!\!\!\!\AxiomC{$\SET{X}=\SET{x} \cf W=w \hspace{8pt} \SET{X}=\SET{x}\cf \gamma$} \RightLabel{$\cf\!\!\textsf{Cmp}${\scriptsize (1)}}\UnaryInfC{$(\SET{X}=\SET{x} \land W=w) \cf \gamma$}\DisplayProof}\\
 \multicolumn{2}{|l@{}|}{ ~~~\AxiomC{$\SET{X}=\SET{x}\cf \bot$}\RightLabel{$\cf\!\bot\textsf{E}$}\UnaryInfC{$\varphi$}\DisplayProof\hfill \AxiomC{}\noLine\UnaryInfC{}\RightLabel{$\bot\!\cf\!\textsf{E}${\scriptsize(2)}}\UnaryInfC{$(\SET Y=\SET y\land X=x\land X=x')\cf \varphi$}\DisplayProof}
\\
\AxiomC{}\noLine\UnaryInfC{$X=x\wedge X=x\wedge\SET{Y}=\SET{y} \cf \varphi$}\RightLabel{$\cf\!\!\textsf{Ctr}$}\UnaryInfC{$X=x\wedge\SET{Y}=\SET{y} \cf \varphi$}\noLine\UnaryInfC{}\DisplayProof&\multirow{2}{*}{ \AxiomC{$\SET{X}=\SET{x} \cf \varphi$}\AxiomC{$[\varphi]$}\noLine\UnaryInfC{$\vdots$}\noLine\UnaryInfC{$\psi$} \RightLabel{$\cf\!\!\textsf{Sub}$}\BinaryInfC{$\SET{X}=\SET{x} \cf \psi$}\noLine\UnaryInfC{}\noLine\UnaryInfC{}
\DisplayProof}\\
 \AxiomC{$X=x\wedge\SET{Y}=\SET{y} \cf \varphi$}\RightLabel{$\cf\!\!\textsf{Wk}$}\UnaryInfC{$X=x\wedge X=x\wedge\SET{Y}=\SET{y} \cf \varphi$}\DisplayProof&
\\
 \multicolumn{2}{|@{}l@{}|}{\!\!\!\AxiomC{}\noLine\UnaryInfC{$\SET{X}=\SET{x}\cf \varphi$} \AxiomC{}\noLine\UnaryInfC{$\SET{X}=\SET{x}\cf \psi$} \RightLabel{{\small$\cf\!\!\wedge\textsf{I}$}} \BinaryInfC{$\SET{X}=\SET{x}\cf \varphi\land \psi$} \DisplayProof\hfill
 \AxiomC{}\noLine\UnaryInfC{$\SET{X}=\SET{x}\cf \varphi\vee \psi$}\RightLabel{{$\cf\!\!\vee\textsf{Dst}$}}\doubleLine\UnaryInfC{$(\SET{X}=\SET{x}\cf \varphi)\vee(\SET{X}=\SET{x}\cf \psi)$} \DisplayProof}\\
\AxiomC{}\noLine\UnaryInfC{$\SET{X}=\SET{x} \cf (\SET{Y}=\SET{y} \cf \varphi)$}\RightLabel{{\small$\cf\!\textsf{Extr}$}\,{\scriptsize (3)}}\UnaryInfC{$(\SET{X'}=\SET{x'} \land \SET{Y}=\SET{y}) \cf \varphi$}\DisplayProof
& \hfill\AxiomC{}\noLine\UnaryInfC{$(\SET{X}=\SET{x} \land \SET{Y}=\SET{y}) \cf \varphi$}\RightLabel{$\cf\!\textsf{Exp}$\,{\scriptsize (4)}}\UnaryInfC{$\SET{X}=\SET{x} \cf (\SET{Y}=\SET{y} \cf \varphi)$} \DisplayProof\\
\multicolumn{2}{|l@{}|}{~~\AxiomC{}\noLine\UnaryInfC{$\neg(\SET{X}=\SET{x}\cf \alpha)$}\RightLabel{$\neg\!\cf\!\textsf{E}$}\UnaryInfC{$\SET{X}=\SET{x}\cf \neg\alpha$} \DisplayProof\hfill\AxiomC{}\noLine\UnaryInfC{$X_1 \leadsto X_2  \hspace{10pt} \dots \hspace{10pt} \dots  \hspace{10pt} X_{k-1} \leadsto X_k$}\RightLabel{$\textsf{Recur}$\,{\scriptsize (5)}}\UnaryInfC{$\neg(X_k \leadsto X_1)$} \DisplayProof}\\
\multicolumn{2}{|l@{}|}{{\footnotesize (1) $\gamma$ is $\cf$-free. (2)  $x\neq x'$.  (3) $\SET X=\SET x$ is consistent, $\SET{X'} = \SET X \setminus \SET Y$, $\SET{x'} = \SET x \setminus \SET y$.  (4)  $\SET{X}\cap \SET Y = \emptyset$.  }}\\[-8pt]
\multicolumn{2}{|l@{}|}{{\footnotesize (5) $X_i\neq X_j$ ($i\neq j$), and $X\leadsto Y$ (meaning ``$X$ causally affects $Y$'') is defined as: \,}}\\[-6pt]
\multicolumn{2}{|l@{}|}{{\footnotesize$\displaystyle
X\leadsto Y :=\bigvee \Big\{\SET Z = \SET z \cf \big((X=x \cf Y=y) \land (X=x' \cf Y=y')\big)$}}\\[-10pt]
\multicolumn{2}{|l@{}|}{{\footnotesize\hfill$\mid \SET Z\subseteq \Dom\setminus\{X,Y\},\SET z\in \Ran(\SET Z), x,x'\in \Ran(X), y,y'\in \Ran(Y),x\neq x',y\neq y'\Big\}$.\quad\quad}}\\
\hline
\end{tabular}
\end{center}
}
\end{itemize}

\end{df}

Note that the above system is parametrized with the signature $\sigma$, and the rules with double horizontal lines are invertible. 
We write $\Gamma\ded\varphi$ (or simply $\Gamma\vdash\varphi$ when $\sigma$ is clear from the context) if the formula $\varphi$ can be derived from $\Gamma$ by applying the rules in  the above system. 

It is easy to verify that all rules in our system are sound for recursive  (generalized) causal teams. 
%
%
%
%
The axioms and rules in the Hilbert system of \cite{BarSan2019} are either included or derivable in our natural deduction system, as shown in the next proposition

\begin{prp}\label{derivable_rules_co}
The following are derivable in the system for $\CO$: 
\begin{enumerate}[(i)]
\item\label{derivable_rules_co_mp}  $\alpha,\neg\alpha\vee\varphi\vdash \varphi$ (weak modus ponens) 
\item\label{derivable_rules_co_unq} $\SET{X}=\SET{x}\cf Y=y\vdash \SET{X}=\SET{x}\cf Y\neq y'$ (Uniqueness)
\item\label{derivable_rules_co_extr_conj} $\SET{X}=\SET{x}\cf \varphi\land \psi\vdash \SET{X}=\SET{x}\cf \varphi$ (Extraction)
\item\label{derivable_rules_co_extr_neg} $\neg(\SET{X}=\SET{x}\cf \alpha)\dashv\vdash\SET{X}=\SET{x}\cf \neg\alpha$
\item\label{derivable_rules_cf_df} $\bigvee_{y\in \Ran(Y)} (\SET{X}=\SET{x}\cf Y=y)$ (Definiteness) 
\end{enumerate}
\end{prp}
\begin{proof}
Item (\ref{derivable_rules_co_mp}) follows  from $\neg\textsf{E}$ and $\vee\textsf{E}$. 
Items (\ref{derivable_rules_co_unq}),(\ref{derivable_rules_co_extr_conj}) follow from $\textsf{ValUnq}$,  $\wedge\textsf{E}$ and $\cf\!\textsf{Sub}$. 
For item (\ref{derivable_rules_co_extr_neg}), the left to right direction follows from $\neg\!\cf\!\textsf{E}$. For the other direction, we first derive by applying $\cf\!\wedge\textsf{I}$, and $\cf\!\bot\textsf{E}$  that
\begin{center}$
\SET{X}=\SET{x}\cf \neg\alpha, \SET{X}=\SET{x}\cf \alpha\vdash \SET{X}=\SET{x}\cf \neg\alpha\wedge\alpha\vdash\SET{X}=\SET{x}\cf\bot  
\vdash \bot
$\end{center}
Then, by $\neg\textsf{I}$ we conclude that $\SET{X}=\SET{x}\cf \neg\alpha\vdash \neg(\SET{X}=\SET{x}\cf \alpha)$.

For item (\ref{derivable_rules_cf_df}), we first derive by $\cf\!\textsf{Eff}$ that $\vdash \SET{X}=\SET{x}\cf X=x$, where  $X=x$ is an arbitrary equation from $\SET{X}=\SET{x}$. By \textsf{ValDef} we also have that $\vdash\bigvee_{y\in \Ran(Y)}Y=y$. Thus, we conclude by applying $\cf\!\!\textsf{Sub}$ that $\vdash\SET{X}=\SET{x}\cf \bigvee_{y\in \Ran(Y)}Y=y$, which then implies that $\vdash\bigvee_{y\in \Ran(Y)} (\SET{X}=\SET{x}\cf Y=y)$ by $\cf\!\!\vee\textsf{Dst}$.
\end{proof}

\begin{teo}[Completeness]\label{completeness_co}
Let $\Delta\cup\{\alpha\}$ be a set of $\CO$-formulas. Then 
\(\Delta\vdash\alpha\iff\Delta\models^{c/g}\alpha.\)
\end{teo}
\begin{proof}
Since our system derives all axioms and rules of the Hilbert system of \cite{BarSan2019}, the completeness of our system over causal teams follows from that of \cite{BarSan2019}.
The completeness of the system over generalized causal teams follows from the fact that $\Delta\models^{c}\alpha$ iff $\Delta\models^{g}\alpha$, given by Corollary \ref{gct_sem_cons_2_ct}.
\end{proof}


\section{Extensions of $\Co$} \label{SECCOMPLETE}

\subsection{Expressive power of $\COv$ and $\COd$}

In this section, we characterize the expressive power of $\COv$ and $\COd$ over causal teams. We show that both logics characterize all nonempty causally downward closed   team properties up to causal equivalence, and the two logics are thus expressively equivalent. 
An analogous result can be obtained for generalized causal teams, but we omit it due to space limitations.


\begin{teo}\label{TEOCHARCODGCT}\label{TEOCHARCDU}
Let  $\K$ be a nonempty (finite) class of causal teams over some signature $\sigma$. Then the following are equivalent:
\begin{enumerate}[(i)]
\item $\K$ is  causally downward closed and closed under equivalence. 
\item $\K$ is definable by a $\COV$-formula.
\item $\K$ is definable by a $\COD$-formula.
\end{enumerate}
\end{teo}

By Theorems \ref{TEOGENDW} and \ref{PROPEQUIV},  for every $\COV$- or $\COD$-formula $\varphi$, the set $\K_\varphi$ is nonempty, causally downward closed and closed under causal equivalence. Thus items (ii) and (iii) of the above theorem imply item (i). 
Since dependence atoms $\dep{\SET X}{Y}$ are definable in $\COV$ (see (1)),  item (iii)  implies item (ii). It then suffices to show that item (i) implies item (iii). In this proof, we make essential use of a formula $\Xi^T$ that  resembles, in the causal setting, a similar formula introduced in \cite{YanVaa2016}  in the pure team setting.

Given any causal team $T=(T^-,\G)$ over $\sigma$, let $\overline{T}=(\ASS\setminus T^-,\G)$ and $\G_0\in \FUNr$ be such that $[\G_0]=[\G]$. If $T^-\neq\emptyset$ and $|T^-|=k+1$, define a $\COD$-formula
\begin{center}$\displaystyle\Xi^T:= (\chi_k \lor \Theta^{\overline{T}})\vee\bigvee_{\mathcal F\in \FUNr\setminus \{\G_0\}}\Phi^\F,$\end{center}
where the formula $\chi_k$ is defined inductively as
\begin{center}
$\displaystyle\chi_0=\bot,~~\chi_1= \bigwedge_{V\in \Dom}\con{V},~~\text{and }\chi_{k}=\chi_1\lor\underbrace{\cdots}_{k\text{ times}}\lor \chi_1~ (k>1)$.
\end{center}

\begin{lm}\label{LEMNOTSUB_ct}
Let $S,T$ be two causal teams over some signature $\sigma$ with $T^-\neq\emptyset$. 
Then, 
$S\models \Xi^T \iff \text{for all } R: T\approx R\text{ implies }R\nsubseteq  S.
$
\end{lm}
\begin{proof}
First, observe that the formula $\chi_k$ characterize the  cardinality of  causal teams $S$, in the sense that 
\begin{center}
\hfill$S\models\chi_k$ ~iff~ $|S^-|\leq k$.\hfill (2)
\end{center}
Indeed, clearly, $S\models\chi_0$ iff $S^-=\emptyset$, $S\models\chi_1$ iff $|S^-|\leq1$, and  for $k>1$, $S\models\chi_k$ iff $S=S_1\cup\dots\cup S_k$ with each $S_i\models\chi_1$ iff 
$|S^-|\leq k$.



Now we prove the  lemma. Let $S=(S^-,\mathcal H)$. ``$\Longrightarrow$": Suppose $S\models\Xi^T$. If $\mathcal H\not\sim \G$, then $T=(T^-,\G)\approx (T^-,\G')=R$ implies $\G'\neq \mathcal H$, thereby $R\nsubseteq S$. Now, suppose $\mathcal H\sim \G\sim \G_0$. If $S^-=\emptyset$, then since $T^-\neq\emptyset$, the statement holds. If $S^-\neq \emptyset$, then by Lemma \ref{LEMPHIF}(ii), we know that no nonempty subteam of $S$ satisfies $\bigvee_{\mathcal F\in \FUNr\setminus\{\G_0\}}\Phi^\F$.
Thus there exist $S_1,S_2\subseteq S$ such that $S_1^-\cup S_2^- = S^-$, 
\begin{center}\hfill$\displaystyle S_1\models \chi_k\text{ and }S_2\models \Theta^{\overline{T}}
.$\hfill (3)\end{center}
By (2), the first clause of the above implies that $|S_1^-|\leq k$. Since $|T^-|=k+1>k$, this means that $T^-\setminus S_1^-\neq \emptyset$. By Lemma \ref{COROCHARSUB_ct}, it follows from the second clause of (3) and the fact that $S_2\models \Phi^\G$ (given again by Lemma \ref{LEMPHIF}(ii)) that $S_2\approx R_0\subseteq\overline{T}$ for some $R_0$. Thus, $T^-\cap S_2^-=\emptyset$.  Altogether, we conclude that $T^-\nsubseteq S^-$. 
 Thus, for any $R$ such that $R\approx T$, we must have that $R^-=T^-\nsubseteq S^-$, thereby $R\nsubseteq S$.

``$\Longleftarrow$": Suppose $T\approx R$ implies $R\not\subseteq  S$ for all $R$. If $\mathcal H\not\sim \G\sim \G_0$, then by Lemma \ref{LEMPHIF}(ii) we have that $S\models\bigvee_{\mathcal F\in \FUNr\setminus\{\G_0\}}\Phi^\F$, thereby $S\models\Xi^T$, as required. Now, suppose $\mathcal H\sim \G$. The assumption then implies that  $T^-\nsubseteq S^-$. Let  $S_1=(S^-\cap T^-,\mathcal H)$ and $S_2=(S^-\setminus T^-,\mathcal H)$. Clearly, $S^-=S_1^-\cup S_2^-$, and it suffices to show that (3) holds. By definition we have that $S_2^-\subseteq (\overline{T})^-$, which implies the second clause of (3) by Lemma \ref{LEMCHARSUB_ct}.
To prove the first clause of (3), by (2) it suffices to verify that $|S_1^-|\leq k$. Indeed, since $T^-\nsubseteq S^-$, we have that $T^-\supsetneq S^-\cap T^-=S_1^-$. Hence, $|S_1^-|<|T^-|=k+1$, namely, $|S_1^-|\leq k$.
\end{proof}

Now we are in a position to give the proof of our main theorem of the section.

\begin{lateproof}{Theorem \ref{TEOCHARCODGCT}}
%
We prove that item (i) implies item (iii). Let $\K$ be a nonempty finite class of causal teams as described in item (i). Since $\K$ is nonempty and causally downward closed, all causal teams over $\sigma$ with empty team component belong to $\K$. Thus, every causal team $T\in \CT\setminus\K$ has a nonempty team component, where $\CT$ denotes the (finite) set of all causal teams over  $\sigma$. 
Now, define 
\(\displaystyle\varphi=\bigwedge_{T\in \CT \setminus \K}\Xi^T.\)
We show that $\K=\K_\varphi$. 

For any $S\notin \K$, i.e., $S\in \CT \setminus \K$, since $S\subseteq S$ and $S^-\neq\emptyset$, by Lemma \ref{LEMNOTSUB_ct} we have that $S\not\models \Xi^S$. Thus $S\not\models\varphi$, i.e., $S\notin \K_\varphi$.
Conversely, suppose $S\in \K$. Take any $T\in \CT \setminus \K$. If $T\approx R\subseteq S$ for some $R$, then since $\K$ is closed under equivalence and causally closed downward, we must conclude that $T\in \K$, which is a contradiction. Thus, by Lemma \ref{LEMNOTSUB_ct}, $S\models \Xi^T$. Hence $S\models\varphi$, i.e., $S\in \K_\varphi$.
\end{lateproof}

\subsection{Axiomatizing $\COv$  over generalized causal teams}

In this section, 
we introduce a sound and complete system of natural deduction for $\COV$, which extends of the system  for $\CO$, and can also be seen as
a variant of the systems for propositional dependence logics introduced in \cite{YanVaa2016}.

\begin{df}\label{COV_system_df}
The system of natural deduction for $\COV$ 
consists of all rules of the system of $\CO$ (see Definition \ref{co-system-df}) together with the following rules, where note that in the rules $\vee\textsf{E}$, $\neg\textsf{I}$, $\neg\textsf{E}$, $\textsf{RAA}$ and $\neg\!\!\cf\!\textsf{I}$ from Definition \ref{co-system-df} the formula $\alpha$ ranges over $\CO$-formulas only:
\begin{itemize}
\item Additional rules for $\vee$:\\
{\normalfont
\renewcommand{\arraystretch}{1.8}
\begin{tabular}{|C{0.45\linewidth}C{0.45\linewidth}|}
\hline
 \multirow{2}{*}{\AxiomC{}\noLine\UnaryInfC{$\varphi\vee\psi$} \RightLabel{$\vee\textsf{Com}$}\UnaryInfC{$\psi\vee\varphi$}\DisplayProof\quad\AxiomC{}\noLine\UnaryInfC{$(\varphi\vee\psi)\vee\chi$} \RightLabel{$\vee\textsf{Ass}$}\UnaryInfC{$\varphi\vee(\psi\vee\chi)$}\DisplayProof}&\AxiomC{$\varphi\vee\psi$}\AxiomC{}\noLine\UnaryInfC{[$\varphi$]}\noLine\UnaryInfC{$\vdots$}\noLine\UnaryInfC{$\chi$} \RightLabel{$\vee\textsf{Sub}$} \BinaryInfC{$\chi\vee\psi$}\noLine\UnaryInfC{}\DisplayProof\\\hline
 \end{tabular}
 %
}

\item Rules for $\vvee$:\\
{\normalfont
\setlength{\tabcolsep}{4pt}
\renewcommand{\arraystretch}{1.8}
\setlength{\extrarowheight}{1pt}
\begin{tabular}{|C{0.45\linewidth}C{0.48\linewidth}|}
\hline
 \multirow{2}{*}{\AxiomC{$\varphi$} \RightLabel{$\vvee\!\textsf{I}$}\UnaryInfC{$\varphi\vvee\psi$}\DisplayProof~\AxiomC{$\varphi$} \RightLabel{$\vvee\!\textsf{I}$}\UnaryInfC{$\psi\vvee\varphi$}\DisplayProof}&\AxiomC{$\varphi\vvee\psi$}\AxiomC{}\noLine\UnaryInfC{[$\varphi$]}\noLine\UnaryInfC{$\vdots$}\noLine\UnaryInfC{$\chi$} \AxiomC{}\noLine\UnaryInfC{[$\psi$]}\noLine\UnaryInfC{$\vdots$}\noLine\UnaryInfC{$\chi$}\RightLabel{$\vvee\!\textsf{E}$} \TrinaryInfC{$\chi$}\DisplayProof\\
  \multicolumn{2}{|c|}{\AxiomC{$\varphi\vee(\psi\vvee\chi)$}\RightLabel{$\vee\!\vvee\!\textsf{Dst}$}\UnaryInfC{$(\varphi\vee\psi)\vvee(\varphi\vee\chi)$} \noLine\UnaryInfC{}\DisplayProof
 \AxiomC{$\SET X = \SET x \cf \psi \vvee \chi$}\RightLabel{$\cf\!\!\!\vvee$\!\textsf{Dst}}\UnaryInfC{$(\SET X = \SET x \cf \psi) \vvee (\SET X = \SET x \cf \chi)$} \noLine\UnaryInfC{}\DisplayProof}\\\hline
\end{tabular}
}
\end{itemize}
\end{df}

The rules in our system are clearly sound. We now proceed to prove the completeness theorem.
An important lemma for the theorem states that every $\COV$-formula $\varphi$ is provably equivalent to the $\vvee$-disjunction of a (finite) set of $\CO$-formulas. Formulas of this type are called \emph{resolutions} of $\varphi$ in \cite{Cia2016b}.


\begin{df}
Let $\varphi$ be a $\COV$-formula. Define the set $\mathcal R(\varphi)$ of its \textbf{resolutions} inductively as follows:
\begin{itemize}
\item $\mathcal R(X=x) = \{X=x\}$,
\item $\mathcal R(\neg\alpha) = \{\neg\alpha\}$, 
\item $\mathcal R(\psi\land\chi) = \{\alpha\land\beta \ | \ \alpha \in \mathcal R(\psi), \beta \in \mathcal R(\chi) \}$,
\item $\mathcal R(\psi\lor\chi) = \{\alpha\lor\beta \ | \ \alpha \in \mathcal R(\psi), \beta \in \mathcal R(\chi)\}$,
\item $\mathcal R(\psi\vvee\chi) = \mathcal R(\psi)\cup\mathcal R(\chi)$,
\item $\mathcal R(\SET X=\SET x\cf\varphi) = \{\SET X=\SET x\cf\alpha \ | \ \alpha \in \mathcal R(\varphi)\}$.
\end{itemize}
\end{df}

\noindent The set $\mathcal R(\varphi)$ is clearly a finite set of $\CO$-formulas. 

%

\begin{lm} \label{normal_form_derivable}
For any formula $\varphi\in \COV$, we have that 
\(\varphi\dashv\vdash\bigvvee \mathcal R(\varphi).\)
\end{lm}

\begin{proof}
We prove the lemma by induction on $\varphi$. 
If $\varphi$ is $X=x$ or $\neg\alpha$ for some  $\CO$-formula $\alpha$, then $\mathcal R(\varphi)=\{\varphi\}$, and 
$\varphi\dashv\vdash\bigvvee \mathcal R(\varphi)$ holds trivially. 



Now, suppose $\psi\dashv\vdash\bigvvee \mathcal R(\psi)$ and $\chi\dashv\vdash\bigvvee \mathcal R(\chi)$. If $\varphi=\psi\wedge\chi$, observing that $\theta_0\wedge(\theta_1\vvee\theta_2)\dashv\vdash(\theta_0\wedge\theta_1)\vvee(\theta_0\wedge\theta_2)$ (by $\vvee\textsf{E}$,$\vvee\textsf{I}$,$\wedge\textsf{I}$,$\wedge\textsf{E}$), we derive by $\vvee\textsf{I}$,$\vvee\textsf{E}$ that 
%
\begin{center}\(\psi\wedge\chi\dashv\vdash \big(\bigvvee \mathcal R(\psi)\big)\wedge\big(\bigvvee \mathcal R(\chi)\big)\dashv\vdash \bigvvee\{\alpha\wedge \beta\mid \alpha\in \mathcal R(\psi),\beta\in \mathcal R(\chi)\}\dashv\vdash \bigvvee \mathcal R(\psi\wedge\chi).\)
\end{center}

If $\varphi=\psi\vee\chi$, we have analogous derivations using the fact that $\theta_0\vee(\theta_1\vvee\theta_2)\dashv\vdash(\theta_0\vee\theta_1)\vvee(\theta_0\vee\theta_2)$ (by $\vee\vvee\textsf{Dst}$,$\vvee\textsf{I}$,$\vvee\textsf{E}$ and 
$\vee\textsf{Sub}$) and $\vvee\textsf{I}$, $\vvee\textsf{E}$.

If $\varphi=\psi\vvee\chi$, then by applying $\vvee\textsf{I}$ and $\vvee\textsf{E}$, we have that
\begin{center}\(\psi\vvee\chi\dashv\vdash \big(\bigvvee \mathcal R(\psi)\big)\vvee\big(\bigvvee \mathcal R(\chi)\big)\dashv\vdash
\bigvvee \big(\mathcal R(\psi) \cup  R(\chi)\big)
\dashv\vdash \bigvvee \mathcal R(\psi\vvee\chi).\)\end{center}

If $\varphi=\SET X=\SET x\cf\psi$, then \allowdisplaybreaks
\begin{center}$\begin{array}{rlcr}
\SET X=\SET x\cf\psi&\dashv\vdash \SET X=\SET x\cf \bigvvee \mathcal R(\psi)&~~&(\cf\!\!\textsf{Sub})
\\
&\dashv\vdash \bigvvee \{\SET X=\SET x\cf \alpha\mid \alpha\in \mathcal R(\psi)\}&& (\cf\!\!\vvee\textsf{Dst}\text{, and }\vvee\textsf{I}, \cf\!\textsf{Sub},\vvee\textsf{E})
\\
\end{array}$\end{center}
\begin{center}$\begin{array}{rlcr}
~~&\dashv\vdash \bigvvee \mathcal R(\SET X=\SET x\cf\psi).&&\textcolor{white}{(\cf\!\!\vvee\textsf{Dst}\text{, and }\vvee\textsf{I}, \cf\!\textsf{Sub},\vvee\textsf{E})}
\end{array}$\end{center}
\end{proof}

\begin{teo}[Completeness]\label{TEOCOMPLCOU}
Let $\Gamma\cup\{\psi\}$ be a set of $\COV$-formulas. Then 
\(\Gamma\vdash\psi\iff\Gamma\models^{g}\psi.\)
\end{teo}
\begin{proof}
We prove the ``$\Longleftarrow$'' direction. Observe that there are only finitely many classes of causal teams of signature $\sigma$. Thus, any set of $\COV$-formulas is equivalent to a single $\COV$-formula, and it then suffices to prove the statement for $\Gamma = \{\varphi\}$.  

Now suppose $\varphi\models\psi$. Then by Lemma \ref{normal_form_derivable} 
and soundness we have that
\(
\bigvvee \mathcal R(\varphi) \models \bigvvee\mathcal R(\psi).
\)
Thus, for every $\gamma\in \mathcal R(\varphi)$, 
\(
\gamma \models^g \bigvvee\mathcal R(\psi),
\)
which further implies, 
by 
Lemma \ref{splitting_prop}, that there is an $\alpha_\gamma\in \mathcal R(\psi)$ such that $\gamma \models \alpha_\gamma$. Since $\gamma,\alpha_\gamma$ are $\CO$-formulas, and the system for $\COV$ extends that for $\CO$, we obtain by  the completeness theorem of $\CO$ (Theorem \ref{completeness_co})  that $\gamma \vdash \alpha_\gamma$. 
Applying $\vvee$\textsf{I} and Lemma \ref{normal_form_derivable}, we obtain
 \(
\gamma \vdash \bigvvee \mathcal R(\psi) \vdash \psi
\)
for each $\gamma\in \mathcal R(\psi)$. Thus, by Lemma \ref{normal_form_derivable} and repeated applications of $\vvee$\textsf{E}, we conclude that
 \(
\varphi\vdash\bigvvee \mathcal R(\varphi) \vdash \psi.
\)
\end{proof}


\subsection{Axiomatizing $\COv$  over causal teams}\label{AXCOCT}
The  method for the completeness proof of the previous subsection cannot be used for causal team semantics, as it makes essential use of the disjunction property of $\vvee$, which fails over causal teams. However, since causal teams can be regarded as a special case of generalized causal teams,  all the rules in the system for $\COv$ over generalized causal teams are also sound over causal teams. 
We can then axiomatize $\COv$ over causal teams by extending the system of $\COv$ for generalized causal teams with an axiom characterizing the property of being uniform, i.e. ``indistinguishable'' from a causal team. 




\begin{df}
The system for $\COV$ over causal teams consists of all rules of $\COV$ over generalized causal teams (Def. \ref{COV_system_df}) plus the following axiom:
{\normalfont
\begin{center}
\begin{tabular}{|C{0.96\linewidth}|}
\hline
\AxiomC{}\noLine\UnaryInfC{} \RightLabel{\textsf{Unf}}\UnaryInfC{$\bigvvee_{\mathcal F\in \FUN} \Phi^{\mathcal F}$}\noLine\UnaryInfC{}\DisplayProof\\\hline
\end{tabular}
\end{center}
}
\end{df}
By Theorem \ref{LEMPHIF}(ii), the axiom \textsf{Unf} is clearly sound over causal teams.

\begin{lm}\label{LEMCTMODELS}
For any set $\Gamma\cup\{\psi\}$ of $\COV$-formulas, 
$\displaystyle
\Gamma\models^{c} \psi$ iff  $\displaystyle\Gamma,\bigvvee_{\mathcal F\in \FUN} \Phi^{\mathcal F}\models^{g} \psi. 
$
\end{lm}

\begin{proof}
$\Longleftarrow$: Suppose $T\models^{c} \Gamma$ for some causal team $T$. Consider the generalized causal team $T^{g}$ generated by $T$.  By Lemma \ref{LEMIDENTIFY},  $T^{g}\models^{g}\Gamma$. Since $T^{g}$ is uniform,  Corollary \ref{LEMCHARCT} gives that $T^{g}\models^{g} \bigvvee_{\mathcal F\in \FUN}\Phi^{\mathcal F}$. Then, by assumption, we obtain that $T^{g}\models^{g}\psi$, which, by Lemma \ref{LEMIDENTIFY} again, implies that $T\models^{c}\psi$.

$\Longrightarrow$:   Suppose $T\models^{g}\Gamma$ and $T\models^{g}\bigvvee_{\mathcal F\in \FUN}\Phi^{\mathcal F}$ for some generalized causal team $T$. By Corollary \ref{LEMCHARCT} we know that $T$ is uniform. Pick $(t,\mathcal F)\in T$.  
%
Consider the generalized causal team $S=\{(s,\F)\mid s\in T^-\}$. Observe that $T\approx S$. Thus, by Theorem \ref{PROPEQUIV}, we have that $S\models^{g}\Gamma$, which further implies, by Lemma \ref{LEMIDENTIFY}(ii), that $S^c\models^{c}\Gamma$. Hence, by the assumption we conclude that $S^c\models^c\psi$. Finally, by applying Lemma \ref{LEMIDENTIFY}(ii) and Theorem \ref{PROPEQUIV} again, we obtain $T\models^g\psi$.
%
%
%
\end{proof}

\begin{teo}[Completeness]\label{cmpl_thm_covvee}
Let $\Gamma\cup\{\psi\}$ be a set of $\COV$-formulas. Then   
\(\Gamma\models^{c}\psi \iff \Gamma\vdash^{c}\psi.\)
\end{teo}

\begin{proof}
Suppose $\Gamma\models^{c}\psi$. By Lemma \ref{LEMCTMODELS}, we have that $\Gamma,\bigvvee_{\mathcal F\in \FUN} \Phi^{\mathcal F}\models^{g} \psi$, which implies that $\Gamma,\bigvvee_{\mathcal F\in \FUN} \Phi^{\mathcal F}\vdash\psi$, by the completeness theorem (\ref{TEOCOMPLCOU}) of the system for $\COV$ over generalized causal teams.  Thus, $\Gamma\vdash \psi$ by axiom \textsf{Unf}. 
\end{proof}

\subsection{Axiomatizing $\COd$}
We briefly sketch the analogous axiomatization results for the language $\COd$ over both semantics. 

    Over generalized causal teams, the system for $\COD$ consists of all the rules of the system for $\CO$ (Definition \ref{co-system-df}) together with $\vee\textsf{Com}$, $\vee\textsf{Ass}$, $\vee\textsf{Sub}$ (the ``additional rules for $\lor$'' from Definition \ref{COV_system_df}) and   the new rules 
    for dependence atoms defined below: 
    \begin{center}
\normalfont
\renewcommand{\arraystretch}{1.8}
\begin{tabular}{|C{0.96\linewidth}|}
\hline
\AxiomC{$X=x$} \RightLabel{$\textsf{DepI}_0$}\UnaryInfC{$\depc{X}$}\DisplayProof\AxiomC{}\noLine\UnaryInfC{[$\depc{X_1}$]} \AxiomC{}\noLine\UnaryInfC{$\dots$} \AxiomC{}\noLine\UnaryInfC{[$\depc{X_n}$]}\noLine\TrinaryInfC{}
\branchDeduce
\DeduceC{$\depc{Y}$}\RightLabel{$\mathsf{DepI}$}\UnaryInfC{$\dep{X_1,\dots,X_n}{Y}$}\DisplayProof\\[-6pt]
\AxiomC{$\varphi$}\AxiomC{}\noLine\UnaryInfC{$\forall x\in Ran(X)$}\noLine\UnaryInfC{[$\varphi[X=x/\depc{X}]$]}\noLine\UnaryInfC{$\vdots$}\noLine\UnaryInfC{$\psi$}\RightLabel{$\textsf{Dep}_0\textsf{E}$ {\footnotesize($\ast$)}}\BinaryInfC{$\psi$}\noLine\UnaryInfC{} \DisplayProof~
\AxiomC{$\dep{X_1,\dots,X_n}{Y}$}\AxiomC{$\depc{X_1}\,\dots\, \depc{X_n}$} \RightLabel{$\textsf{DepE}$}\BinaryInfC{$\depc{Y}$}\DisplayProof\\[-10pt]
 \multicolumn{1}{|L{0.96\linewidth}|}{{\footnotesize ($\ast$) $\varphi[X=x/\depc{X}]$ stands for the formula obtained by replacing \emph{a specific occurrence} of $\depc{X}$ in $\varphi$ with $X=x$.}}\\\hline
\end{tabular}
\end{center}
    These rules for dependence atoms generalize the corresponding rules in the pure team setting as introduced in \cite{YanVaa2016}. The completeness theorem of the system can be proved by generalizing the corresponding arguments in \cite{YanVaa2016}. 
Analogously to the case for \COv, 
in this proof we use the fact that every formula $\varphi$ is (semantically) equivalent to a formula $\bigvvee_{i\in I}\alpha_i$ in disjunctive normal form, where each $\alpha_i$ 
is a $\CO$-formula obtained from $\varphi$ by replacing every dependence atom $\dep{\SET X}{Y}$ by a formula $\bigvee_{\SET x\in \Ran(\SET X)}(\SET X=\SET x\wedge Y=y)$ with $y$ ranging over all of $\Ran(Y)$. The disjunctive formula $\bigvvee_{i\in I}\alpha_i$ is not in the language of $\COd$, but we can prove in the system of $\COd$ (by applying the additional rules in the table above) that $\alpha_i\vdash\varphi$ ($i\in I$), and that 
\begin{center}
$\Gamma,\alpha_i\vdash\psi \text{ for all }i\in I\Longrightarrow \Gamma,\varphi\vdash\psi.$
\end{center}
These mean {\em in effect} that ``$\varphi\dashv\vdash\bigvvee_{i\in I}\alpha_i$''.
The completeness theorem for $\COd$ is then proved using essentially the same strategy as that for $\COv$ (Theorem \ref{TEOCOMPLCOU}).

Over causal teams, using the same method as in the previous section, the complete system for $\COD$ can be defined as an extension of the above generalized causal team system with two additional axioms \textsf{1Fun} and \textsf{NoMix}, defined as follows: 

\begin{center}
\normalfont
\begin{tabular}{|C{0.96\linewidth}|}
\hline
\AxiomC{}\noLine\UnaryInfC{}\noLine\UnaryInfC{} \RightLabel{\textsf{1Fun} {\footnotesize(1)}}\UnaryInfC{$\displaystyle\bigwedge_{V\in \Dom} \Big(\beta_{\mathrm{En}}(V) \supset (\bigwedge_{\SET w\in \SET W_V} \SET W_V = \SET w \cf \con{V} )\Big)$}\DisplayProof\\
\AxiomC{}\noLine\UnaryInfC{}\noLine\UnaryInfC{} \RightLabel{\textsf{NoMix} {\footnotesize(2)}}\UnaryInfC{$\displaystyle\bigwedge_{V\in Dom} \bigwedge\{\Xi_*^{\{a,b\}}\mid (a,b)\in \SEM^2,\,\{a\}\models \beta_{\mathrm{En}}(V), \,\{b\}\not\models \beta_{\mathrm{En}}(V)\}$}\DisplayProof\\
\\
 \multicolumn{1}{|L{0.96\linewidth}|}{{\footnotesize (1) $\SET W_V=\Dom\setminus\{V\}$, and $\displaystyle\beta_{\mathrm{En}}(V):=\bigvee_{X\in\SET W_V} \beta_{\mathrm{DC}}(X,V)$, where each $\beta_{\mathrm{DC}}(X,V)$ is the  $\CO$-formula from \cite{BarSan2019} expressing the property ``$X$ is a direct cause of $V$'':}}\\ 
  \multicolumn{1}{|L{0.96\linewidth}|}{{\footnotesize\quad$\displaystyle
\beta_{\mathrm{DC}}(X,V):=\bigvee\big\{(\SET Z = \SET z \land X=x) \cf V=v,~(\SET Z = \SET z \land X=x') \cf V=v'$}}\\
  \multicolumn{1}{|C{0.96\linewidth}|}{{\footnotesize\quad\quad\quad\quad\quad\quad\quad\quad$\mid x,x'\in \Ran(X),\,v,v'\in \Ran(V),\,\SET Z=\Dom\setminus\{X, V\},\,\SET z\in \Ran(\SET Z),\,x\neq x', v\neq v'\big\}.
$ 
 }}\\
 [-6pt]\\
  \multicolumn{1}{|L{0.96\linewidth}|}{{\footnotesize (2) $\Xi_*^{\{a,b\}}$ is defined otherwise the same as $\Xi^{\{a,b\}}$ except that $\chi_1$ is redefined as}} \\
  \multicolumn{1}{|C{0.96\linewidth}|}{{\footnotesize $\displaystyle\chi_1:=\bigwedge_{V\in Dom} \big(\con{V}\land \bigwedge_{\SET w \in Ran(\SET W_V)}(\SET W_V = \SET w \cf \con{V})\big)$.}}\\\hline
\end{tabular}
\end{center}

\noindent 
The axiom \textsf{1Fun} states that the endogenous variables are governed by a unique function; 
 the axiom \textsf{NoMix} guarantees that all members of the generalized causal team agree on what is the set of endogenous variables. 
Together, these two additional axioms characterize the {\em uniformity} of the generalized causal team in question (or they  are equivalent to the formula \textsf{Unf} in \COV), thus allow for a completeness proof along the lines of Section \ref{AXCOCT}.

\section{Conclusion}

We have answered the main questions concerning the expressive power and the existence of deduction calculi for the languages that were proposed in \cite{BarSan2018} and \cite{BarSan2019}, and which involve both (interventionist) counterfactuals and (contingent) dependencies. In the process, we have introduced a generalized causal team semantics, for which we have also provided natural deduction calculi. We point out that our calculi are sound only for \emph{recursive} systems, i.e., when the causal graph is acyclic. The general case (and special cases such as the ``Lewisian'' systems considered in \cite{Zha2013}) will require a separate study. We point out, however, that each of our deduction systems can be adapted to the case of \emph{unique-solution} (possibly generalized) causal teams by replacing the \textsf{Recur} rule with an inference rule that expresses the \emph{Reversibility} axiom from \cite{GalPea1998}.

Our work shows that many  methodologies developed in the literature on team semantics can be adapted to the generalized semantics and, to a lesser extent, to causal team semantics. 
On the other hand, a number of peculiarities emerged that set apart these semantic frameworks from the usual team semantics: for example, the failure of the disjunction property over causal teams. 
We believe the present work may provide guidelines for the investigation of further notions of dependence and causation in causal team semantics and its variants.

\end{document}